\documentclass[11pt]{article}
\usepackage{geometry}
\usepackage[compress]{cite}
\usepackage{amssymb}
\usepackage{xcolor}
\usepackage{mathrsfs}
\usepackage{amsfonts}
\usepackage{amsmath}
\numberwithin{equation}{section}
\allowdisplaybreaks[4]
\newtheorem{Theorem}{Theorem}[section]
\newtheorem{Proposition}[Theorem]{Proposition}
\newtheorem{Corollary}[Theorem]{Corollary}
\newtheorem{Lemma}[Theorem]{Lemma}
\newtheorem{Example}[Theorem]{Example}

\newtheorem{Definition}[Theorem]{Definition}

\def\rank{\mathrm{rank}\,}

\def\bbZ{\mathbb{Z}}
\def\bbR{\mathbb{R}}

\def\VN{V_m|_{[N_1, N_2]}}
\def\supp{\mathrm{supp}\ }

\begin{document}


\title{Local and Global Phaseless Sampling in Real Spline Spaces\thanks{
 This work was partially supported by the
National Natural Science Foundation of China (11371200, 11525104 and  11531013).}
}

\author{Wenchang Sun \\
\footnotesize School of Mathematical Sciences and LPMC,  Nankai University,       Tianjin~300071, China \mbox{}\hskip -1em \mbox{}\\
\footnotesize Email: \,\,\,  sunwch@nankai.edu.cn}

\date{}

\maketitle

\begin{abstract}
We study the recovery of functions in real spline spaces from unsigned sampled values.
We consider two types of recovery. The one is to recover functions locally from finitely many unsigned samples.
And the other is to recover functions on the whole line from infinitely many unsigned samples.
In both cases, we give characterizations for a sequence of distinct points to be a phaseless sampling sequence, at which
any nonseparable  function is determined up to a sign on an interval or on the whole line by its unsigned sampled values.
Moreover, for the case of local recovery, we also study the almost phaseless sampling and give a necessary and sufficient condition for a sequence of points to
admit local recovery for almost all functions.
\end{abstract}


\section{Introduction and Main Results}

The sampling theory is one of the most powerful results in signal analysis.
It says that when a function satisfies certain conditions,  it can be recovered from sampled values.
In practice, it might happen that we have only intensity measurements. That is, sampled values are phaseless.
To recover a function from intensity measurements, we have to study the problem of phase retrieval,
which arises in the recovery of functions given the magnitude of its Fourier transform.
We refer to the review paper \cite{Shechtman2015} for an introduction on this topic.

Recently, many works have been done on the phase retrieval problem for general frames since
Balan, Casazza  and Edidin \cite{BCE2006} introduced the concept of phaseless reconstruction
in the setting of frame theory.
For the case of finite-dimensional, various aspects to this problem which include the uniqueness and the stability of solutions
were well studied \cite{BCE2006,Bandeira2014,Bandeira2014b,Bodmann2016,CCPW2016,CLNZ2016,Edidin2017,GKK2017,HES2016,IVW2017,IVW2016,QSHHS2016,QBP2016,TEM2016}.
Further generalizations including
norm retrieval  \cite{bahmanpour2014,Casazza2017}
and phase retrieval from projections \cite{Cahill2015,Edidin2017} were also studied.

For the case of infinite-dimensional, the problem becomes very different \cite{Alaifari2016,alaifari2016b,Mallat2015,Pohl2014,Pohl2014b,Shenoy2016,Thakur2011,yang2013}.
In particular, it was shown by Cahill, Casazza and Daubechies \cite{JCD2016} that
phase retrieval is never uniformly stable in the infinite-dimensional case.
And in \cite{CCSW2016,Cheng2017}, Chen, Cheng, Jiang, Sun and Wang  studied phase retrieval of real-valued functions in shift-invariant spaces.
They gave some density results on the sequence of phaseless sampling points and studied the  stability of phase retrieval.
They showed that not all functions in such spaces can be recovered from intensity measurements.
In particular, separable functions can not be recovered up to a sign.
Recall that a function $f$ in a function space $H$ is said to be separable if $f=f_1+f_2$ for some $f_1, f_2\in H\setminus\{0\}$
with $f_1(x)f_2(x) = 0$.

In this paper, we study the problem of phaseless sampling in real spline spaces.
Specifically, let
\[
   \varphi_m = \chi^{}_{[0,1]} * \cdots * \chi^{}_{[0,1]}  \quad  ( m+1 \mathrm{\,\, terms}),\,\, m\ge 1
\]
be the $m$-degree B-spline and
\[
  V_m = \Big\{\sum_{k\in\bbZ} c_n\varphi_m(\cdot-n):\,c_n\in\bbR\Big\}
\]
be the real spline space generated by $\varphi_m$.
Note that $\varphi_m$ is compactly supported. The series is well defined on $\bbR$ for any real sequence $\{c_n:\,n\in\bbR\}$.
The problem is to recover a function $f$ from its unsigned sampled values $|f(x_i)|$, where $\{x_i:\,  i\in I\}$ is a sequence of sampling points.

One of the fundamental problems for the phaseless sampling
is to determine sequences of sampling points which admit a local or global recovery of functions in given function spaces.
In this paper, we study the characterization of sequences of sampling points at which we can recover any nonseparable function  in
spline spaces from unsigned samples.

We consider two types of phaseless sampling problems for functions in $V_m$.
The one is local phaseless sampling, i.e., to recover functions on an finite interval from finitely many unsigned samples.
Specifically, given two integers $N_1<N_2$ and a sequence of distinct points
$E:=\{x_i:\, 1\le i\le N\}\subset [N_1, N_2]$, we search for conditions on $E$ which
guarantee  a local reconstruction of $f$ on $[N_1, N_2]$ up to a sign from its unsigned sampled values $|f(x_i)|$, $1\le i\le N$.

Local phaseless sampling is  practically useful since measured signals are always time limited.
For convenience, we introduce the following definition.

Let $\VN$  be the restriction of $V_m$
 on $[N_1, N_2]$. That is, $f\in\! \VN$ if and only if $f = g \cdot \chi_{[N_1,N_2]}$ for some $g\in V_m$.

\begin{Definition}
We call a sequence $E\subset [N_1, N_2]$ consisting of distinct points  a \emph{(local) phaseless sampling sequence} for $\VN$
if
any nonseparable function $f\in\VN$ is determined up to a sign by its unsigned sampled values on $E$.
\end{Definition}

We give a necessary and sufficient condition for a sequence to be a local phaseless sampling sequence.

\begin{Theorem}\label{thm:PR}
Let $N_1<N_2$ be integers.
A sequence $E\subset [N_1, N_2]$ consisting of distinct points is a phaseless sampling sequence for $\VN$ if and only if it satisfies the followings,
\begin{eqnarray}
\# E &\ge& 2(N_2-N_1+m)-1,          \label{eq:PR:1}               \\
\# (E\cap [N_1, N_1+k)) &\ge& 2k+m-1,         \quad   1\le k\le N_2-N_1,     \label{eq:PR:2}  \\
\# (E\cap (N_2-k, N_2]) &\ge& 2k+m-1,  \quad   1\le k\le N_2-N_1,\label{eq:PR:3} \\
\# (E\cap(n_1,n_2)) &\ge& 2(n_2-n_1)-1,\quad   N_1\le n_1<n_2\le N_2,           \label{eq:PR:4}
\end{eqnarray}
where $\# E$ denotes the cardinality of a sequence $E$.

Moreover, if $E$ meets the above conditions and $|f_1(x)|=|f_2(x)|$ on $E$ for some $f_1,f_2\in V_m$,
then $|f_1(x)| = |f_2(x)|$ on $[N_1,N_2]$.
\end{Theorem}

As shown in \cite{CCSW2016,Cheng2017}, separable functions can not be recovered up to a sign from intensity measurements. Nevertheless, we see from
Theorem~\ref{thm:PR} that if $E$ is a phaseless sampling sequence for $\VN$, then for any $f\in \VN$, $|f(x)|\cdot \chi_{[N_1,N_2]}(x)$ is determined uniquely by its
sampled values over $E$.

In \cite{Fickus2014}, Fickus,   Mixon,   Nelson  and   Wang studied the problem of almost phase retrieval for general frames. Here ``almost''
means that for almost all functions in a finite-dimensional function space, it is possible to recover the function from intensity measurements.
When almost phase retrieval is considered, we need only very few measurements.
In this paper, we study the problem of almost phaseless sampling on $\VN$.
We give a characterization for a sequence to be an almost phaseless sampling sequence for $\VN$,
which means that almost all functions in $\VN$ are uniquely determined up to a sign by their unsigned samples on such a sequence.

The other type of problem we are considered is the global phaseless sampling.
That is, to recover a function on the whole line from its unsigned sampled values.
Again, we study the construction of sequences $E\subset \bbR$ which admit a global phase retrieval, that is,
any nonseparable function  in $V_m$ is determined up to a sign by its unsigned sampled values on $E$.
We call such sequences phaseless sampling sequences for $V_m$.

A characterization of phaseless sampling sequences for $V_m$ reads as followings.

\begin{Theorem}\label{thm:global}
Let $E\subset \bbR$ be a sequence of distinct points.
For $m\ge 2$, $E$ is a phaseless sampling sequence for $V_m$ if and only if it satisfies the following
(P1) and (P2).
\begin{enumerate}
\item[\upshape (P1)] For any integers $n_1<n_2$, $\#(E\cap(n_1,n_2))  \ge 2(n_2-n_1)-1$;
\item[\upshape (P2)] For any integer $n_0$, there exist integers $n_2>n_1\ge n_0$ and $i_1<i_2\le n_0$ such that
    $\#(E\cap [n_1, n_2]) \ge 2(n_2-n_1+m)-1$
    and $\#(E\cap [i_1, i_2]) \ge 2(i_2-i_1+m)-1$.
\end{enumerate}

For $m=1$,
$E$ is a phaseless sampling sequence for $V_1$ if and only if it satisfies (P1) and
\begin{enumerate}
\item [\upshape (P2')]  there exists an increasing sequence of integers $\{n_k:\, k_1\le k\le k_2\}$, which contains at least one point,
such that  $\#(E\cap[n_k-1,n_k])\ge 3$ for $k_1\le k\le k_2$ and
$\#(E\cap(n,n+1))=2$ for $n\not\in[n_{k_1}, n_{k_2}-1]$ if $k_1>-\infty $ or $k_2<\infty$.
\end{enumerate}

Moreover, if $|f_1(x)|=|f_2(x)|$ on $E$ for some $f_1,f_2\in V_m$
and $E$ meets (P1) and (P2) for $m\ge 2$ or (P1) and (P2') for $m=1$,
then $|f_1(x)| = |f_2(x)|$ for any $x\in\bbR$.
\end{Theorem}

Again, although separable functions can not be recovered up to a sign from intensity measurements, we see from
Theorem~\ref{thm:global} that if $E$ is a phaseless sampling sequence for $V_m$, then for any $f\in V_m$, $|f(x)|$ is determined uniquely by its
sampled values over $E$.

The paper is organized as follows.
In Section 2, we study  the problem of almost phaseless sampling in $\VN$ and give a necessary and sufficient condition
for a sequence to be an almost phaseless sampling sequence.
And in Sections 3 and 4, we give proofs of Theorem~\ref{thm:PR}  and Theorem~\ref{thm:global}, respectively.
In Section 5, we present some examples to illustrate the main results.

\section{Almost Phaseless Sampling in Spline Spaces}

In this section, we study the local recovery of almost all functions in spline spaces from phaseless sampled values.

\begin{Definition}
We call   $E = \{x_i:\, 1\le i\le N\}\subset [N_1, N_2]$ an almost phaseless sampling sequence for $\VN$
 if
for any $  f \in \VN \setminus H_0$,
we can reconstruct $f$ up to a sign from the unsigned sample sequence $\{|f(x_i) |:\,1\le i\le N\}$,
where $H_0$ consists of finitely many proper subspaces of $\VN$ and therefore is of Lebesgue measure zero.
\end{Definition}

The main result in this section is the following characterizations of almost phaseless sampling sequences.

\begin{Theorem} \label{thm:almost}
Let $E \subset [N_1, N_2]$ be a sequence of distinct numbers.
Then $E$ is an almost phaseless sampling sequence for $\VN$
  if and only if it satisfies the following conditions,
\begin{eqnarray}
&& \#E \ge N_2-N_1+m+1, \label{eq:c0a}\\
&&  \#\big(E \cap [N_1, N_1+k)\big) \ge k+1, \quad 1\le k\le N_2-N_1, \label{eq:c1a}\\
&&  \#\big(E \cap (N_2-k, N_2]\big) \ge k+1, \quad 1\le k\le N_2-N_1,\label{eq:c2a}\\
&& \#\big(E \cap (n_1, n_2)\big) \ge n_2-n_1-m+1, \quad N_1\le n_1 < n_2 \le N_2.\label{eq:c3a}
\end{eqnarray}
\end{Theorem}

Note that the set consisting of all separable functions in $\VN$ is of $(N_2-N_1+m)$-dimensional Lebesgure measure zero (see Lemma~\ref{Lm:sep}).
One might ask if it is possible to recover all nonseparable functions with an almost phaseless sampling sequence?
The answer is unfortunately negative. In fact, we see from the characterization for local phaseless sampling sequences (Theorem~\ref{thm:PR})
that if $E$ is only an almost phaseless sampling sequence, then many nonseparable functions are unrecoverable from its unsigned sampled values.

Before giving a proof of Theorem~\ref{thm:almost}, we introduce some results on the almost phase retrieval for general frames.

We call a frame $\{f_i:\, 1\le i\le N\}$  for $\mathbb R^n$ \emph{almost phase retrievable} if
for any $  f \in \mathbb R^n \setminus  E_0$,
we can reconstruct $f$ up to a sign from the sequence of unsigned frame coefficients $\{|\langle f, f_i\rangle |:\,1\le i\le N\}$,
where $ E_0$ consists of finitely many proper subspaces of $\bbR^n$ and therefore is of Lebesgue measure zero.

In \cite{Fickus2014,Zhongwei}, some necessary and sufficient conditions for a frame to be almost phase retrievable were given.
Here we give some further characterizations for  almost phase retrievable frames.

Denote by $S_N=\{(s_1,\ldots,s_N):\, s_1=1, s_i=\pm 1, i\ge 2\}$.
For an $n\times N$ matrix $A$, $A^*$ is the transpose of $A$ and $\mathcal N(A)$ is the null space of $A$, i.e.,
$\mathcal N(A) = \{x\in\bbR^N:\, Ax=0\}$.

Since frames for $\bbR^n$ are equivalent to $n\times N$ matrices with rank $n$, for convenience, we
also say a matrix is almost phase retrievable if its  column vectors form an
almost phase retrievable frame.

Let  $\varphi_m(\cdot-n)|_{[N_1,N_2]}$
 be the restriction of  $\varphi_m(\cdot-n)$
 on $[N_1, N_2]$.
That is,
\[
\varphi_m(x-n)|_{[N_1,N_2]} = \left\{
    \begin{array}{ll}
      \varphi_m(x-n), & x\in [N_1, N_2], \\
      0, & \mathrm{otherwise}.
    \end{array}
 \right.
\]
It is easy to see that
$\{\varphi_m(\cdot-n)|_{[N_1,N_2]}:\,N_1-m\le n\le N_2-1\}$ is a basis for $\VN$ \cite[Lemma 4]{SunZhou09}.
Therefore, $\dim \VN = N_2-N_1+m$.

\begin{Theorem} \label{thm:almost PR frames}
Suppose that $A$ is an $n\times N$ matrix whose column vectors form a frame for  $\mathbb R^n$, $n\ge 2$.
Let $s, s'\in S_N$ and $s \ne s'$.
Then the followings are equivalent.

\begin{enumerate}
\item \label{item:1}
$A$   is almost phase retrievable.

\item \label{item:2}
$D_s A^* \bbR^n \ne D_{s'} A^* \bbR^n$.

\item \label{item:3}
 $\mathcal N(A D_s) \ne \mathcal N(A D_{s'})$.

\item \label{item:4}
$
  \rank(AD_s)  <
  \rank\begin{pmatrix}
    AD_s \\
    AD_{s'}
  \end{pmatrix}
$.

\item \label{item:5}
$
  \rank(MD_s)  <
  \rank\begin{pmatrix}
    MD_s \\
    MD_{s'}
  \end{pmatrix}
$,
where $M$ is an $(N-n)\times N$ matrix whose null space is the range of $A^*$.

\end{enumerate}
\end{Theorem}

\begin{proof}
Denote the column vectors of $A$ by $\varphi_1$, $\ldots$, $\varphi_N$.

(\ref{item:1})$\Rightarrow$(\ref{item:2}).\,\, Since $s\ne s'$, there is some $i_0\ge 2$ such that $s_{i_0} \ne s'_{i_0}$.
Without loss of generality, we assume that $s_{i_0}=1$ and $s'_{i_0}=-1$.
Let $E = \{x\in\bbR^n:\, \langle x, \varphi_1\rangle = 0$   or $\langle x, \varphi_{i_0}\rangle = 0 \}$.
Then $E$ is of measure zero.

Assume that $D_s A^* \bbR^n = D_{s'} A^* \bbR^n$. Then for any $x\in\bbR^n$, there is some $x'\in \bbR^n$
such that $D_sA^* x = D_{s'} A^* x'$. Hence
\[
  \langle x, \varphi_1\rangle = \langle x', \varphi_1\rangle
  \quad \mathrm{and}\quad
  \langle x, \varphi_{i_0}\rangle = -\langle x', \varphi_{i_0}\rangle.
\]
It follows that for $x\not\in E$, $x\ne \pm x'$, which contradicts with (\ref{item:1}).

(\ref{item:2})$\Rightarrow$(\ref{item:1}). For $s,s'\in S_N$ with $s\ne s'$, define
\[
  E_{s,s'} = \{x\in\bbR^n:\, \mbox{there is some $x'\in\bbR^n$ such that } D_sA^* x = D_{s'} A^* x'\}.
\]
Since $D_sA^*\bbR^n \ne D_{s'}A^*\bbR^n$,
$D_sA^*\bbR^n \cap D_{s'}A^*\bbR^n$ is a proper subspace of $D_sA^*\bbR^n$. Consequently,
$E_{s,s'}$ is a proper subspace of $\bbR^n$. Let
\[
  E_0 = \bigcup_{s,s'\in S_N, s\ne s' } E_{s,s'}.
\]
Then $E_0$ is the union of finitely many proper subspaces of $\bbR^n$. For any $x\in \bbR^n\setminus E_0$,
if there is some $x'\in\bbR^n$ such that
$|\langle x, \varphi_i\rangle| = |\langle x', \varphi_i\rangle|$,
then there is some $s\in S_N$ such that $A^* x = D_s A^* x'$ or $A^* x = -D_s A^* x'$.
Hence $x=x'$ or $x=-x'$. In other words, $A$ is almost phase retrievable.

(\ref{item:2})$\Leftrightarrow$(\ref{item:3}). It follows from the fact that $\mathcal N(A D_s) = (D_sA^*\bbR^n)^\bot$.

(\ref{item:3})$\Leftrightarrow$(\ref{item:4}). Observe that
$\rank(AD_s) = \rank(AD_{s'})$.
If
\begin{equation}\label{eq:e1}
  \rank(AD_s)  =
  \rank\begin{pmatrix}
    AD_s \\
    AD_{s'}
  \end{pmatrix},
\end{equation}
then there exists some invertible $n\times n$ matrix $P$ such that
$AD_s = P AD_{s'}$. And vice versa.  It follows that (\ref{eq:e1}) is equivalent to
$\mathcal N(A D_s) = \mathcal N(A D_{s'})$. This proves the equivalence of (\ref{item:3}) and (\ref{item:4}).

(\ref{item:1})$\Leftrightarrow$(\ref{item:5}).\,\,
See \cite[Theorem 8]{Zhongwei}.
\end{proof}

Let $A$ be an $n\times N$ matrix with rank $n$.
We say that  $A$ is weak full spark if its rank remains unchanged when any one of its columns is removed.

For the case of $N=n+1$, we show that almost phase retrievable frames are equivalent to weak full spark matrices.

\begin{Theorem} \label{thm:t1}
Let the hypothesis be as in Theorem~\ref{thm:almost PR frames}.  Then we have

\begin{enumerate}
\item If $A$ is almost phase retrievable, then $A$ is weak full spark.
\item Conversely, if $A$ is weak full spark and $N=n+1$, then $A$ is almost phase retrievable.
\end{enumerate}
\end{Theorem}

\begin{proof}
(i).\,\, Denote the column vectors of $A$ by $\varphi_1$, $\ldots$, $\varphi_N$.
If $A$ is not weak full spark, then there is some $1\le i_0\le N$ such that $\varphi_{i_0}$  can not be written as a linear combination of
the others. Consequently, if
\[
  \sum_{i=1}^N c_i \varphi_i = 0,
\]
then we have $c_{i_0}=0$.
Let $s = (1,\ldots, 1)$. If $i_0=1$, then set $s'=(1,-1,\ldots,-1)$. Otherwise, set
$s'=(1,\ldots, -1,\ldots,1)$ (only the $i_0$-th entry is $-1$). In both cases, we have $s\ne s'$
and
$\mathcal N(AD_s) = \mathcal N(AD_{s'})$.
By Theorem~\ref{thm:almost PR frames}(\ref{item:3}),
$A$ is not almost phase retrievable, which contradicts with the hypothesis.
Hence $A$ is weak full spark.

(ii).\, Assume that $A$ is not  almost phase retrievable.
By Theorem~\ref{thm:almost PR frames}(\ref{item:3}), there exist some $s, s'\in S_N$ with $s\ne s'$ such that
$\mathcal N(AD_s) = \mathcal N(AD_{s'})$.

Let $I_1 = \{1\le i\le N: s_i = s'_i\}$
and $I_2$ be the complement of $I_1$.
Let $c = (c_1, \ldots, c_N)^* \in \bbR^N$ be such that
\begin{equation}\label{eq:e2}
   AD_s c = 0.
\end{equation}
Then we have $AD_{s'}c=0$. Hence
\begin{equation}\label{eq:e3}
  \sum_{i\in I_1} c_i s_i\varphi_i  = \sum_{i\in I_2} c_i s_i\varphi_i = 0.
\end{equation}
There are two cases.

(a) For  $l=1$ or $l=2$, the solution of (\ref{eq:e2}) satisfies that
\[
    c_i=0,\qquad i\in I_l.
\]

In this case, for $i\in I_l$, $\varphi_i$ can not be written as a linear combination of other vectors
which is impossible since $A$ is weak full spark.

(b) There is a solution of (\ref{eq:e2}) such that neither $\{c_i:\, i\in I_1\}$
nor $\{c_i:\, i\in I_2\}$ is a sequence of zeros.

In this case, we see from (\ref{eq:e3}) that
\[
  \rank(\{\varphi_i:\, i\in I_1\}) \le \#I_1 -1
  \quad \mathrm{and}\quad
  \rank(\{\varphi_i:\, i\in I_2\}) \le \#I_2 -1.
\]
Hence
\[
  \rank(\{\varphi_i:\, 1\le i\le N\}) \le \#I_1+\#I_2-2 = N-2 = n-1,
\]
which is possible since $\{\varphi_i:\, 1\le i\le N\}$ is a frame for $\bbR^n$.
\end{proof}

Next we consider the reconstruction of functions in spline spaces from phaseless sampled values.
We begin with some results on local sampling in spline spaces.

\begin{Definition}
We call $E=\{x^{}_k:\,1\le k\le K\}$
  a (local) sampling sequence for $\VN$
if $E\subset [N_1, N_2]$ and there is a sequence of functions $\{S_k:\,1\le k\le K\}$ such that
\[
     f(x) = \sum_{k=1}^{K} f(x^{}_k) S_k(x), \qquad \forall f\in V_m,\,\,  x\in [N_1, N_2].
\]
\end{Definition}

Based  on the celebrated Sch\"onberg-Whitney Theorem\cite{Sch},
the following characterization of local sampling sequences for  spline spaces was proved in \cite{SunZhou09}.

\begin{Proposition} \label{prop:local sampling}
A sequence $E\subset [N_1, N_2]$ of distinct points  is a sampling sequence for $\VN$ if
and only if it satisfies the following conditions,
\begin{eqnarray}
&& \#E \ge N_2-N_1+m, \label{eq:c0}\\
&&  \#\big(E \cap [N_1, N_1+k)\big) \ge k, \quad 0\le k\le N_2-N_1, \label{eq:c1}\\
&&  \#\big(E \cap (N_2-k, N_2]\big) \ge k, \quad 0\le k\le N_2-N_1,\label{eq:c2}\\
&& \#\big(E \cap (n_1, n_2)\big) \ge n_2-n_1-m, \quad N_1\le n_1 < n_2 \le N_2.\label{eq:c3}
\end{eqnarray}
\end{Proposition}

Applying Theorem~\ref{thm:almost PR frames} to the local phaseless sampling in spline spaces,
we get the following characterization of almost phaseless sampling sequences.

\begin{Lemma} \label{Lm:L0}
Let $E=\{x_i:\,1\le i\le N\}\subset [N_1, N_2]$ be a sequence of distinct numbers
and $K:=N_2-N_1+m\le N$.
Define the  $K\times N$ matrix $\Phi$ by
\begin{equation}\label{eq:Phi}
   \Phi = \Big[\varphi_m(x_j-n)\Big]_{N_1-m\le n\le N_2-1, 1\le j\le N}.
\end{equation}
Then the following items are equivalent.
\begin{enumerate}
\item $E$ is an almost phaseless sampling sequence for $\VN$.

\item
$\{D_s\Phi^*\bbR^K:\, s\in P_N \}$ consists of distinct  $K$-dimensional subspaces.

\item
$\{\mathcal N(\Phi  D_s): s\in\! P_N\! \}$ consists of distinct  $(N-K)$-dimensional subspaces.
\end{enumerate}
\end{Lemma}

\begin{proof}
The equivalence of (ii) and (iii) is obvious. We only need to show that (i) and (ii) are equivalent.

(i) $\Rightarrow$ (ii).
Let $E$  be an almost phaseless sampling sequence.
First, we show that $\rank(\Phi) = K$. Assume on the contrary that
\[
  \rank(\Phi) < K.
\]
Then we can find some $c\in R^K\setminus\{0\}$ such that $\Phi^* c=0$.
Let
\[
  f_0(x) = \sum_{n=N_1-m}^{N_2-1} c_n \varphi_m(x-n).
\]
Then we have $f_0\ne 0$ and $f_0(x_i)=0$ for $1\le i\le N$. Consequently, for $f\ne -(1/2) f_0$,
we have $f \ne \pm (f+f_0)$ but $|f(x_i)| = |f(x_i) + f_0(x_i)|$. Hence
we can not recover $f$ from  intensity measurements, which contradicts with the assumption.

Set $F = (f(x_1),\ldots, f(x_N))^*$
and $c = (c_{N_1-m}, \ldots, c_{N_2-1})^*$.
Then we have
\[
  F = \Phi^* c.
\]
Since $E$ is an almost phaseless sampling sequence and
$\{\varphi_m(\cdot-n)|_{[N_1,N_2]}:\,N_1-m\le n\le N_2-1\}$ is a basis for $\VN$,
column vectors of $\Phi$ form an almost phase retrievable frame for $\bbR^K$.
Now the conclusion follows from Theorem~\ref{thm:almost PR frames}.

(ii) $\Rightarrow$ (i) can be proved similarly.
\end{proof}

\begin{Lemma} \label{Lm:L1}
An almost phaseless sampling sequence for $\VN$ is always a sampling sequence for $\VN$.
\end{Lemma}

\begin{proof}
Let  $E=\{x_i:\,1\le i\le N\}$
be an almost phaseless sampling sequence for $\VN$.
Define $\Phi$   by (\ref{eq:Phi}).
Set $F = (f(x_1),\ldots, f(x_N))^*$
and $c = (c_{N_1-m}, \ldots, c_{N_2-1})^*$.
Then we have
\[
  F = \Phi^* c.
\]
We see from the proof of Lemma~\ref{Lm:L0} that
$\rank(\Phi) = N_2-N_1+m$. Hence $\Phi^*\Phi$ is invertible. Therefore,
\[
  c = (\Phi\Phi^*)^{-1}\Phi F.
\]
It follows that
\[
  f(x) = \sum_{n=N_1-m}^{N_2-1} c_n \varphi_m(x-n)
       = \sum_{k=i}^N f(x_i)  S_i(x), \qquad x\in [N_1, N_2],
\]
where $S_i$ is a linear combination of $\{\varphi_m(x-n):\, N_1-m\le n\le N_2-1\}$.
In other words, $E$ is a local sampling sequence for $\VN$.
\end{proof}

With the above lemma, we  get the minimum cardinality of almost phaseless sampling sequences on $[N_1, N_2]$.

\begin{Corollary} \label{Co:c1}
Suppose that $E$ is  an almost phaseless sampling sequence for $\VN$. Then we have
$\#E \ge N_2-N_1+m+1$.
\end{Corollary}

\begin{proof}
By Lemma~\ref{Lm:L1}, $E$ is a sampling sequence for
$\VN$.
Hence  $\# E \ge N_2-N_1+m$.
If $\# E =  N_2-N_1+m$, then
the matrix $\Phi$ defined in (\ref{eq:Phi}) is an $(N_2-N_1+m) \times (N_2-N_1+m)$ invertible matrix.
By Lemma~\ref{Lm:L0}, $E$ is not a phaseless sampling sequence, which contradicts with
the hypothesis. This completes the proof.
\end{proof}

Next we show that every almost phaseless sampling sequence on $[N_1, N_2]$ contains a subsequence whose cardinality equals to the minimum $N_2-N_1+m+1$.

\begin{Lemma}\label{Lm:L2}
Let $E \subset [N_1, N_2]$ be sequence of distinct numbers which meets (\ref{eq:c0a})--(\ref{eq:c3a}).
Then there is a subsequence $E'\subset E$ such that $ \# E'= N_2-N_1+m+1$ and $E'$
meets (\ref{eq:c1a})--(\ref{eq:c3a}).
\end{Lemma}

\begin{proof}
First, we consider the case of $m=1$.  We see from (\ref{eq:c1a})--(\ref{eq:c3a}) that
\begin{eqnarray*}
 && \#(E\cap [N_1, N_1+1)) \ge 2 \\
 && \#(E\cap (N_2-1, N_2]) \ge 2, \\
 &&\#(E\cap (n, n+1)) \ge 1, \quad N_1\le n\le N_2-1.
\end{eqnarray*}
Hence there is some $E'\subset E$   such that
\begin{eqnarray*}
 && \#(E'\cap [N_1, N_1+1)) = 2 \\
 && \#(E'\cap (N_2-1, N_2]) =2, \\
 &&\#(E'\cap (n, n+1)) = 1, \quad N_1< n< N_2-1.
\end{eqnarray*}
Now we get a subsequence $E'$ as desired.

Next we consider the case of $m\ge 2$. Set $K=\#E$.
Assume that  $K >N_2-N_1+m+1$.
Let
\begin{eqnarray*}
  l_k &=& \#\big(E \cap [N_1, N_1+k)\big) - k-1, \qquad 1\le k \le  N_2-N_1.
\end{eqnarray*}
Then we see from (\ref{eq:c1a})  that
\begin{equation}\label{eq:ab}
l_k \ge 0,   \qquad 1\le k\le N_2-N_1.
\end{equation}
Set
\begin{eqnarray}
k_0 &=& \min\{k\in [1, N_2-N_1]:\,
         l_i\ge 1, k\le i\le N_2-N_1\}. \label{eq:ka}
\end{eqnarray}
We conclude that
\begin{equation}\label{eq:k12}
\#\big(E \cap ( N_2-k, N_2]\big)
   \ge k+2, \quad  N_2-N_1-k_0+1 \le k\le N_2-N_1.
\end{equation}

Note that
\[
  l_{N_2-N_1}  \ge K-1 - (N_2-N_1) -1 \ge m \ge 1.
\]
We have $1\le k_0 \le N_2-N_1$.
Since
\[
  \#(E\cap (N_1, N_2]) \ge K-1 \ge N_2-N_1+m+1,
\]
(\ref{eq:k12}) is true if $k_0=1$.

Next we suppose that $k_0>1$.
We see from (\ref{eq:ka}) that  $l_{k_0-1} =0$. Hence
\[
     \#\big(E \cap [N_1, N_1+k_0 - 1)\big) = k_0.
\]
Therefore,
\begin{equation}\label{eq:ea1}
     \#\big(E \cap [ N_1+k_0 - 1, N_2]\big) = K-k_0.
\end{equation}
For $N_2-N_1-k_0+1< k \le N_2-N_1$, we see from (\ref{eq:c3a}) and (\ref{eq:ea1}) that
\begin{eqnarray*}
   &&  \#\big(E \cap ( N_2-k, N_2]\big) \\
&=& \#\big(E \cap ( N_2-k,  N_1+k_0-1)\big)
     + \#\big(E \cap [ N_1+k_0 - 1, N_2]\big) \\
&\ge& (N_1-N_2+k+k_0-m) +  (K-k_0) \\
&=& K-(N_2-N_1+m)+k \\
&\ge& k+2.
\end{eqnarray*}
And for $k=N_2-N_1-k_0+1$, we see from (\ref{eq:ea1}) that
\begin{eqnarray*}
 \#\big(E \cap ( N_2-k, N_2]\big)
   &\ge& K-k_0 -1
             = K-N_2+N_1-2+k \ge k+ m \\
   &\ge& k+2.
\end{eqnarray*}
Hence (\ref{eq:k12}) is also true.

Next we show that
\begin{equation}\label{eq:k3}
\#\big(E \cap (N_1+k_0 - 1, N_1+k_0)\big)\ge 1.
\end{equation}

If $k_0=1$, then (\ref{eq:k3}) follows from (\ref{eq:c1a}).
For the case of $k_0>1$,  we see from (\ref{eq:ka}) that $l_{k_0-1}=0$ and $l_k\ge 1$ for $k_0\le k\le N_2-N_1$.
Hence
\begin{eqnarray}
&&    \#\big(E \cap [ N_1+k_0 - 1, N_1+k)\big) \nonumber\\
&=& \#\big(E \cap [ N_1, N_1+k)\big) - \#\big(E \cap [ N_1, N_1+k_0-1)\big)\nonumber \\
&\ge& k+2 - k_0, \qquad k_0\le k\le N_2-N_1.\label{eq:numka}
\end{eqnarray}
By  setting $k=k_0$, we get
\[
\#\big(E \cap [ N_1+k_0 - 1, N_1+k_0)\big)\ge 2.
\]
Hence (\ref{eq:k3}) is true.

Take some $y'\in E \cap ( N_1+k_0 - 1, N_1+k_0)$ and let
$E' = E\setminus\{y'\}$. Then $E'$ meets (\ref{eq:c0a}) -- (\ref{eq:c3a}).

Since $\# E' = K-1 \ge N_2-N_1+m+1$, we need only to show that $E'$
 satisfies (\ref{eq:c1a})-(\ref{eq:c3a}).

For  $1\le k\le k_0-1$, we have
\[
\#\big(E' \cap [ N_1, N_1+k)\big) = \#\big(E \cap [ N_1, N_1+k)\big) \ge k + 1.
\]
And for $k_0 \le k\le N_2-N_1$,
\[
\#\big(E' \cap [ N_1, N_1+k)\big) = \#\big(E \cap [ N_1, N_1+k)\big)-1 = l_k+k \ge k+1.
\]
Hence  (\ref{eq:c1a}) is true.

On the other hand,  for $1\le k\le N_2-N_1+k_0$, we have
\begin{eqnarray*}
\#\big(E' \cap ( N_2-k, N_2]\big)
&=& \#\big(E \cap ( N_2-k, N_2]\big) \\
&\ge& k+1.
\end{eqnarray*}
And for  $N_2-N_1+k_0+1\le k\le N_2-N_1$, we see from (\ref{eq:k12}) that
\begin{eqnarray*}
\#\big(E' \cap ( N_2-k, N_2]\big)
&=& \#\big(E \cap ( N_2-k, N_2]\big)-1  \ge k+1.
\end{eqnarray*}
Hence  (\ref{eq:c2a}) is true.

Now it remains to prove that $E'$ meets (\ref{eq:c3a}).
There are three cases.

Case 1. $n_1\ge N_1+k_0$ or $n_2 \le N_1+k_0-1$. In this case, we have
\[
\#\big(E' \cap ( n_1, n_2)\big) = \#\big(E \cap ( n_1, n_2)\big) \ge n_2-n_1-m+1.
\]

Case 2. $n_1 <N_1+k_0-1$ and $n_2 \ge N_1+k_0$.

We see from (\ref{eq:numka}) that
for $k_0\le k\le N_2-N_1$,
\[
\#\big(E' \cap [N_1+k_0-1, N_1+k)\big)
=\#\big(E \cap [N_1+k_0-1, N_1+k)\big)-1
\ge k-k_0+1.
\]
Hence,  for $n_1 <N_1+k_0-1$ and $n_2 \ge N_1+k_0$,
\begin{eqnarray*}
    \#\big(E' \cap ( n_1, n_2)\big)
\!\!&=& \!\!\#\big(E' \cap ( n_1, N_1+k_0-1)\big)
    + \#\big(E' \cap [N_1+k_0-1, n_2)\big) \\
&\ge&\!\! (N_1+k_0-1-n_1-m+1) + (n_2 - N_1-k_0+1) \\
&=&\!\! n_2-n_1-m+1.
\end{eqnarray*}

Case 3. $n_1 =N_1+k_0-1$ and $n_2 \ge N_1+k_0$.
By (\ref{eq:numka}), we have
\[
\#\big(E' \cap (n_1, n_2)\big)
 = \#\big(E \cap (n_1, n_2)\big)-1
\ge n_2 - n_1 -1 \ge n_2-n_1-m+1.
\]
In all three cases, we show that $E'$ meets (\ref{eq:c3a}).

Repeating the previous arguments again and again, we   get some $E'\subset E$
such that $\# E' = N_2-N_1+m+1$ and
$E'$ meets (\ref{eq:c1a}) -- (\ref{eq:c3a}).
This completes the proof.
\end{proof}

We are now ready to give a proof for Theorem~\ref{thm:almost}.

\begin{proof}[Proof of Theorem~\ref{thm:almost}]
Denote $E=\{x_i:\, 1\le i\le N\}$.
Let $\Phi$ be defined by (\ref{eq:Phi}) and denote its column vectors by  $\Phi_1$, $\ldots$, $\Phi_N$.

First,  we prove the necessity.
Assume that $E$ is an almost phaseless sampling sequence.
By  Corallory~\ref{Co:c1},  $\#E \ge N_2-N_1+m+1$.

We see from  Proposition~\ref{prop:local sampling} and Lemma~\ref{Lm:L1} that
(\ref{eq:c1})-(\ref{eq:c3}) are true.
If (\ref{eq:c1a}) or (\ref{eq:c2a}) is false, then we have
\[
  \#(E\cap[N_1, N_1+k)) = k\quad \mathrm{or}\quad
  \#(E\cap(N_2-k, N_2]) = k.
\]
If $\#(E\cap(n_1,n_2)) \le n_2-n_1-m$ for some $n_1<n_2$, then  $n_2-n_1\ge m$.
We conclude that $E\cap (n_1,n_2) \ne  \emptyset$. Otherwise,
the matrix $\Phi^*$ has the following form,
\[
  \Phi^* =
  \begin{pmatrix}
     A_1   & 0 \\
     0     & A_2
  \end{pmatrix},
\]
where $A_1$ is a $k_1\times (n_1-N_1+m)$ matrix,
$A_2$ is a $k_2\times (N_2-n_1)$ matrix, and $k_1+k_2=N$.
Hence there exist some $s\ne s'$ such that $D_s\Phi^*\bbR^{N_2-N_1+m} = D_{s'}\Phi^*\bbR^{N_2-N_1+m}$.
By Lemma~\ref{Lm:L0}, $E$ is not an almost phaseless sampling sequence, which contradicts with the assumption.
Hence $\#(E\cap(n_1,n_2)) \ge 1$.

It follows from the above arguments that if  $E$ does not meet one of (\ref{eq:c1a})-(\ref{eq:c3a}),
then there is some element in $E$, say $x_N$, such that $E\setminus\{x_N\}$ is not a sampling sequence.

Since $N-1 \ge N_2-N_1+m$, we see from the proof of Lemma~\ref{Lm:L1} that
$\rank(\Phi_1,\ldots, \Phi_{N-1})\!$ $<   N_2-N_1+m$.
But $\rank(\Phi) = N_2-N_1+m$.
Hence $\Phi_N$ is not a linear combination of $\Phi_1$, $\ldots$, $\Phi_{N-1}$.
Therefore, if there exists some $(a_1, \ldots, a_N)\in \bbR^N$  such that
\begin{equation}\label{eq:e5}
  \sum_{i=1}^N a_i \Phi_i = 0,
\end{equation}
then we have $a_N=0$.
It follows that for $s = (1,\ldots, 1, 1)$ and $s' = (1,\ldots, 1, -1)$,
$  \Phi  D_s$ and $ \Phi  D_{s'}$ have the same null space.
By Lemma~\ref{Lm:L0},
 $E$ is not an almost phaseless sampling sequence,
 which contradicts with the assumption.

Next we prove the sufficiency.
Suppose that $E$ meets (\ref{eq:c0a})--(\ref{eq:c3a}).
by Lemma~\ref{Lm:L2}, there is a subsequence $E'\subset E$
such that $K:= \# E'= N_2-N_1+m+1$ and $E'$
meets (\ref{eq:c1a})--(\ref{eq:c3a}).
We see from Proposition~\ref{prop:local sampling} that $E'$ remains a local sampling sequence whenever any one of its elements is removed.
Let $\Phi'$ be defined similarly with (\ref{eq:Phi}).
Then $\Phi'$ is full spark. Hence there exists some $a \in\bbR^K$  with non-zero entries such that
\[
   \Phi' a = 0.
\]
It follows that for $s = (p_1, \ldots, p_K)\in P_K$,
the null space of $\Phi' D_s$  is $\{x D_s a:\, x\in\bbR\}$.
Since entries of $a$ are non-zero, these null spaces are distinct.
By Lemma~\ref{Lm:L0},
 $E'$ is an almost phaseless sampling sequence.
Since $E'\subset E$, we get the conclusion as desired.
\end{proof}

\section{Local Phaseless Sampling in Spline Spaces}

In this section, we give a proof of Theorem~\ref{thm:PR}. We begin with a characterization of separable functions.
The following result coincides with \cite[Theorem II.6]{CCSW2016} and \cite[Corollary 2.6]{Cheng2017}, where the separability for
functions in general shift invariant spaces was established.

\begin{Lemma}\label{Lm:sep}
Let $N_1<N_2$ be integers. Then  a function $f\in \VN$ is separable if and only if $N_2-N_1 \ge 2$ and
$f = \sum_{n=N_1-m}^{N_2-1} c_n \varphi_m(\cdot-n)$ for some $c_{N_1-m},\ldots, c_{N_2-1}\in\bbR$
such that $c_{n_1}\ne 0$, $c_{n_2}\ne 0$ and $c_n=0$, $n_1<n<n_2$ for some $n_1, n_2$ with
$n_2-n_1 \ge m+1$.
\end{Lemma}

\begin{proof}
Since $\{\varphi_m(\cdot-n)|_{[N_1,N_2]}:\, N_1-m\le n\le N_2-1\}$ is a basis for $\VN$, for any $f\in \VN$,
there exist some $c_{N_1-m}$, $\ldots$, $c_{N_2-1}\in\bbR$ such that
$f = \sum_{n=N_1-m}^{N_2-1} c_n \varphi_m(\cdot-n)$.

Necessity.
Assume that $f$ is separable, that is, there exist $f_1, f_2\in \VN\setminus\{0\}$ such that
$f = f_1 + f_2$ and $f_1(x)f_2(x) = 0$ for all $x\in [N_1, N_2]$.
For any $N_1\le n<N_2$, since both $f_1$ and $f_2$ are polynomials on $[n,n+1]$,
one of $f_1$ and $f_2$ must be identical to zero on $[n,n+1]$.

On the other hand, since $f_1, f_2\ne 0$, there exist integers $i_1$ and $i_2$ such that
\[
  f_1|_{[i_2,i_2+1]} \ne 0 \quad \mathrm{and} \quad f_2|_{[i_1,i_1+1]} \ne 0.
\]
Consequently,
\[
  f_1|_{[i_1,i_1+1]} =0 \quad \mathrm{and} \quad f_2|_{[i_2,i_2+1]} = 0.
\]
Hence $N_2-N_1\ge 2$. Without loss of generality, assume that $i_1<i_2$.

Let
\begin{equation}\label{eq:sep:e1}
  i'_1 = \max \{i:\, f_1|_{[i_1,i]}=0\}.
\end{equation}
Then we have $i_1<i'_1 \le i_2$ and $f_1|_{[i'_1, i'_1+1]}\ne 0$. Since $ f_1f_2=0$, we have
$f_2|_{[i'_1, i'_1+1]}=0$.

Suppose that
$f_l = \sum_{n=N_1-m}^{N_2-1} c_n^{(l)} \varphi_m(\cdot-n)$, $l=1,2$.  We see from (\ref{eq:sep:e1}) that
\begin{equation}\label{eq:sep:e2}
  c^{(1)}_{i_1-m} = \ldots = c^{(1)}_{i'_1-1} = 0 \quad \mathrm{and} \quad c^{(1)}_{i'_1} \ne 0.
\end{equation}
On the other hand, we see from $f_2|_{[i'_1,i'_1+1]}=0$ that
\[
  c^{(2)}_{i'_1-m} = \ldots = c^{(2)}_{i'_1} = 0.
\]
Hence $c_{i'_1} \ne 0$ and
\begin{equation}\label{eq:a1}
c_i = c^{(1)}_i + c^{(2)}_i = 0,\qquad i'_1-m \le i\le i'_1-1.
\end{equation}
Since  $f_2|_{[i_1,i_1+1]}\ne 0$,  there is some $i_0$ with $i_1-m\le i_0 \le i_1 < i'_1$ such that
\[
  c^{(2)}_{i_0} \ne 0.
\]
It follows from (\ref{eq:sep:e2}) that  $c_{i_0}\ne 0$. By (\ref{eq:a1}), $i_0<i'_1-m$.
Set $n_1 = \max\{n\in [N_1-m,i'_1-m-1]:\, c_n\ne 0\}$
and $n_2 = i'_1$. Then we get the conclusion as desired.

Sufficiency.
Let $f$ meets the hypothesis.
Set
\[
  f_1 = \sum_{n=N_1-m}^{n_1} c_n \varphi_m(\cdot -n)
  \quad \mathrm{and} \quad
  f_2 = \sum_{n=n_2}^{N_2-1} c_n \varphi_m(\cdot -n).
\]
Then we have $f_1, f_2\ne 0$ and  $f=f_1+f_2$. Since
\[
  \supp f_1 \subset [N_1, n_1+m+1],\qquad
  \supp f_2 \subset [n_2, N_2]
\]
and $n_2-n_1\ge m+1$, we have $f_1f_2=0$.
This completes the proof.
\end{proof}

We see from Proposition~\ref{prop:local sampling} that for a sequence $E\subset [N_1, N_2]$ to be a sampling sequence
for $\VN$, it must satisfy four conditions (\ref{eq:c0})--(\ref{eq:c3}).
The following result shows that whenever $\#E \ge N_2-N_1+m$, it  determines functions in $V_m$ locally.

\begin{Lemma}\label{Lm:samp}
Suppose that  $N_1<N_2$ are integers and  $E\subset [N_1, N_2]$. If $\# E\ge N_2-N_1+m$, then there exist some integers $n_1, n_2\in [N_1, N_2]$
such that $n_1<n_2$ and $E\cap[n_1,n_2]$ is a sampling sequence for $V_m|_{[n_1,n_2]}$.
\end{Lemma}

\begin{proof}
We prove the conclusion with induction on $N_2-N_1$.
We see from Proposition~\ref{prop:local sampling} that it is the case if $N_2-N_1=1$.

Now we assume that the conclusion is true whenever $1\le N_2-N_1\le n$ for some $n\ge 1$.
Let us consider the case of $N_2-N_1=n+1$.
Assume that $E$ is not a sampling sequence for $\VN$.
By Proposition~\ref{prop:local sampling}, there are three cases.

(i).\,\, There is some  $k\in [1, N_2-N_1]$ such that  $\#(E\cap[N_1,N_1+k))<k$.

In this case, we have $k<N_2-N_1$ and  $\#(E\cap [N_1+k, N_2]) \ge N_2-N_1-k+m+1$. Now we see from the inductive assumption that there exist
some integers $n_1,n_2 \in [N_1+k, N_2]$  such that $n_1<n_2$ and
$E\cap[n_1,n_2]$ is a sampling sequence for $V_m|_{[n_1,n_2]}$.

(ii).\,\, There is some $k\in [1, N_2-N_1]$ such that  $\#(E\cap(N_2-k,N_2])<k$.

Similarly to the previous case we can prove the conclusion.

(iii).\,\, $\#(E\cap[N_1,N_1+k))\ge k$ and $\#(E\cap(N_2-k,N_2])\ge k$ for any $1\le k\le N_2-N_1$.

If for any integers $n_1<n_2$ with $[n_1,n_2]\subset [N_1,N_2]$, we have $\#(E\cap(n_1,n_2))\ge n_2-n_1-m$, then we see from Proposition~\ref{prop:local sampling}
that $E$ is a sampling sequence for $\VN$.
If  $\#(E\cap(n_1,n_2))\le n_2-n_1-m-1$ for some $n_1<n_2$, then we have $n_2-n_1\ge m+1$ and $(n_1,n_2)\ne (N_1,N_2)$.
Without loss of generality, we assume that $N_1<n_1$. Then
\begin{eqnarray*}
  \#(E\cap [N_1,n_1]) &=& \#(E\cap [N_1,n_2)) - \#(E\cap (n_1,n_2)) \\
  &\ge& (n_2-N_1)-(n_2-n_1-m-1)\ge n_1-N_1+m+1.
\end{eqnarray*}
By the inductive assumption, there exist some integers $i_1<i_2$ in $[N_1, n_1]$ such that $E\cap[i_1,i_2]$ is a sampling sequence
for $V_m|_{[i_1,i_2]}$.

By induction, the conclusion is true for any $N_1<N_2$.
\end{proof}

To prove Theorem~\ref{thm:PR}, we need the following result on the invertibility of submatrix of (\ref{eq:Phi}).

\begin{Proposition}[{\cite[Theorem 4.65]{Sch}}]\label{prop:Sch}
Let $n_1<\ldots<n_r$ be $r$ integers and $t_1<\ldots<t_r$ be $r$ real numbers.
Then the matrix $[\varphi_m(t_i-n_j)]_{1\le i,j\le r}$ is invertible if and only if
$\varphi_m(t_i-n_i)\ne 0$ for $1\le i\le r$.
\end{Proposition}

\begin{proof}[Proof of the necessity of Theorem~\ref{thm:PR}]
Denote $E=\{x_i:\, 1\le i\le N\}$
and
\[
  J_k = \# (E\cap [N_1, N_1+k)), \qquad 1\le k\le N_2-N_1.
\]
{\color{blue}
First, we prove that
\begin{eqnarray}
\# (E\cap [N_1, N_1+1)) &\ge& m+1,  \quad       \label{eq:PR:2a}  \\
\# (E\cap (N_2-1, N_2]) &\ge& m+1.  \quad       \label{eq:PR:3a}
\end{eqnarray}
}
Assume on the contrary that $J_1\le m$. Define
\begin{eqnarray}
  A_1 &=& [\varphi_m(x_i - n)]_{1\le i\le J_1, N_1-m\le n\le N_2-1}, \label{eq:A1a}\\
  A_2 &=& [\varphi_m(x_i - n)]_{J_1+1\le i\le N, N_1-m\le n\le N_2-1}.\label{eq:A2a}
\end{eqnarray}
Note that
$\rank(A_1) \le J_1\le m$ and
the last $N_2-N_1-1$ columns of $A_1$ are zeros.
There is a solution $c = (c_{N_1-m}, \ldots, c_{N_2-1})^* \in \bbR^{N_2-N_1+m}$  of the equation
\[
  A_1 c = 0
\]
such that $c_{N_1+1} = \ldots = c_{N_2-1} = 1$
and not all of $c_{N_1-m}$, $\ldots$, $c_{N_1}$ are zeros.
Since the first column of $A_1$ is not zero,
not all of $c_{N_1-m+1}$, $\ldots$, $c_{N_1}$ are zeros.

Let $c' = (1+|c_{N_1-m}|, 0,\ldots, 0)^*$. Then we have
\[
  A_2 c' = 0.
\]
Let
\begin{equation}\label{eq:fg}
\left\{
\begin{aligned}
f_1 &= \sum_{n=N_1-m}^{N_2-1} \frac{1}{2}(c_n+c'_n) \varphi_m(\cdot - n),  \\
 f_2 &= \sum_{n=N_1-m}^{N_2-1} \frac{1}{2}(c_n-c'_n) \varphi_m(\cdot - n).
\end{aligned}\right.
\end{equation}
Since there are at most $m-1$ entries of $c\pm c'$ are zeros, by Lemma~\ref{Lm:sep},
both $f_1$ and $f_2$ are nonseparable.  Moreover,  $f_1\pm f_2 \ne 0$ and
\begin{eqnarray*}
   f_1(x_i) &=& -f_2(x_i),\qquad 1\le i\le J_1, \\
   f_1(x_i) &=&  f_2(x_i),\qquad J_1+1\le i\le N.
\end{eqnarray*}
Hence we can not recover $f$ from unsigned samples, which contradicts with the hypothesis.
This proves (\ref{eq:PR:2a}).
Similarly we can prove (\ref{eq:PR:3a}).

{\color{blue} Next we prove (\ref{eq:PR:4}).} We begin with the simple case $n_2-n_1=1$. We have to show that
\begin{equation}\label{eq:n}
\# (E\cap (n,n+1)) \ge 1, \qquad N_1  \le  n\le N_2-1.
\end{equation}
We see from (\ref{eq:PR:2a}) and (\ref{eq:PR:3a}) that (\ref{eq:n}) is true for $n=N_1$ or $N_2-1$.
Now assume that $\# (E\cap (n,n+1)) = 0$ for some $n\in [N_1+1, N_2-2]$.
Define
\begin{eqnarray*}
  A_1 &=& [\varphi_m(x_i - j)]_{x_i \le n, N_1-m\le j\le N_2-1}, \\
  A_2 &=& [\varphi_m(x_i - j)]_{x_i\ge n+1, N_1-m\le j\le N_2-1}.
\end{eqnarray*}
Let
$c = (0,\ldots, 0,1,\ldots,1)^*$
and $c' = (1,\ldots, 1, 0,\ldots,0)^*$,
where the first $n-N_1+m$ entries of $c$ and
the last $N_2-n+m-1$ entries of $c'$ are zeros and the rest are $1$.
Then we have
\[
  A_1 c = A_2 c' = 0.
\]
Let $f_1$ and $f_2$ be defined by (\ref{eq:fg}).
Since there are only $m-1$ entries of $c\pm c'$ are zeros,
with similar arguments we get a contradiction. Hence (\ref{eq:n}) is true.

Now we assume that (\ref{eq:PR:4}) is false. Then there are integers  $i_1<i_2$ with $i_2-i_1\ge 2$
such that $\# (E\cap (i_1, i_2)) \le 2(i_2-i_1)-2$.
Let
\begin{eqnarray}
n_2 &=& \min\{n\in [N_1,N_2]:\, \mbox{ there exists some $n_1<n$ such that } \nonumber  \\
 &&\hskip 30mm  \# (E\cap (n_1, n)) \le 2(n-n_1)-2 \},   \label{eq:n2} \\
n_1 &=& \max\{n\in [N_1,n_2):\, \# (E\cap (n, n_2)) \le 2(n_2-n)-2 \}.\label{eq:n1}
\end{eqnarray}
Then we have
\begin{eqnarray}
  && \hskip -2em n_1  \le n_2-2,  \label{eq:n1n2:1}\\
  &&\hskip -2em \# (E\cap (n_1, n_1+1)) = \# (E\cap (n_1, n_1+1]) = 1, \label{eq:n1 n11}  \\
  &&\hskip -2em \# (E\cap (n_2-1, n_2)) = \# (E\cap [n_2-1, n_2)) = 1, \label{eq:n1n2:3} \\
  &&\hskip -2em  \#(E\cap [n,n+1)) = \#(E\cap (n,n+1]) = 2, \,\, n_1+1\le n\le n_2-2,\label{eq:n1n2:4} \\
    &&\hskip -2em  E\cap(n_1,n_2)\cap \bbZ = \emptyset. \label{eq:n1n2:5}
\end{eqnarray}
Moreover, we see from the definitions of $n_1, n_2$ and (\ref{eq:n1 n11}) that
\begin{align}
  \# (E\cap (n,n+2)) &\ge 3,  & N_1\le n\le n_1-1,  \label{eq:nn2}\\
 \# (E\cap (n,n_1]) &\ge 2(n_1-n), & N_1\le n\le n_1-1.\label{eq:nn1}
\end{align}
It follows 
that
if $n_1>N_1+1$, then
\begin{eqnarray*}
 \#(E\cap [N_1, n_2))
   &= &  \#(E\cap [N_1, N_1+1]) + \#(E\cap (N_1+1, n_1])  \\
   &&\qquad +  \#(E\cap (n_1, n_2))\\
   &\ge& 2(n_2-N_1)+m-3.
\end{eqnarray*}
And the above inequality is also true if $n_1\le N_1+1$.

By (\ref{eq:n}), there is some $E_2\subset E$ such that
\begin{eqnarray}
  &&E\cap[n_2,N_2]\subset E_2, \label{eq:E12:1}\\
  &&\#(E_2 \cap [N_1, n_2) ) = n_2-N_1-1,  \label{eq:E12:2}\\
  &&\#(E_2\cap(n,n+1)) = 1, \qquad N_1 \le n\le n_2-2.\label{eq:E12:3}
\end{eqnarray}
Let $E_1 = E\setminus E_2$.
Define
\begin{eqnarray}
  A_1 &=& [\varphi_m(x_i-l)]_{x_i\in E_1, N_1-m\le l\le N_2-1}, \label{eq:A1}\\
  A_2 &=& [\varphi_m(x_i-l)]_{x_i\in E_2, N_1-m\le l\le N_2-1}.\label{eq:A2}
\end{eqnarray}
Observe that $E_2\cap [n_2-1,n_2) = \emptyset$. The first $n_2-N_1$ columns of $A_2$ has the form
$\begin{pmatrix}
A_{21} \\
0
\end{pmatrix}$,
where
\[
  A_{21} = \begin{pmatrix}
  *  &  *  & ? & \ldots &  &  \\
  0  &  *  & * & \ldots &  &  \\
  0  &  0  & * & \ldots &  &  \\
     &  \ldots &  & & \ldots   &  \\
  0  &  0  & 0 & \ldots & * & * & ?\\
  0  &  0  & 0 & \ldots & 0 & * & * \\
  \end{pmatrix}
\]
is an $(n_2-N_1-1)\times(n_2-N_1)$ matrix.  Here '$*$' stands for a nonzero entry and '$?$' means uncertainty (it is $0$ if $m=1$).
By Proposition~\ref{prop:Sch}, every $(n_2-N_1-1)\times(n_2-N_1-1)$ submatrix of $A_{21}$ is invertible. Hence there exists
some  $ c \in \bbR^{n_2-N_1}$ such that none of its entries is zero and
\[
  A_{21}  c = 0.
\]
By appending $N_2-n_2+m$ zeros to $c$ we get a $c'\in \bbR^{N_2-N_1+m}$, none of whose first $n_2-N_1$ entries is zero,
such that
\[
  A_2 c' = 0.
\]
On the other hand, since $E_1\subset [N_1, n_2)$, we have $A_1 = (A_{11}, 0)$,
where $A_{11}$ has $n_2-N_1+m$ columns and at least $n_2-N_1+m-2\ge n_2-n_1+m-2$ rows.
Note that
\[
\#(E_1\cap (n_1, n_1+1]) = 0 \,\, \mathrm{and}\,\, \#(E_1\cap (n, n+1]) = 1 \,\, \mathrm{for}\, n_1+1\le n\le n_2-1.
\]
$A_{11}$ has the following form,
\[
  A_{11} =
  \begin{pmatrix}
 \ldots &   &      &    & \ldots &   &   & \\
 \ldots & 0 & 0    & 0  & \ldots &   &   & \\
 \ldots & * & *    & 0  & \ldots &   &   & \\
 \ldots &   &\ldots&    & \ldots &   &   & \\
 \ldots &   &      &    & \ldots & * & * & 0  \\
 \ldots &   &      &    & \ldots & ? & * & * \\
  \end{pmatrix},
\]
where the submatrix consisting of the last $n_2-n_1$ columns of $A_{11}$ has only $n_2-n_1-1$ non-zero rows.
Using Proposition~\ref{prop:Sch} again, we  get some $\tilde c\in \bbR^{n_2-N_1+m}$ such that
the first $n_1-N_1+m$ entries of $\tilde c$ are zeros, none of the last $n_2-n_1$ entries is zero,
and $A_{11}\tilde c=0$.  Consequently, there is some $c=(c_{N_1-m},\ldots,c_{N_2-1})^*\in\bbR^{N_2-N_1+m}$ such that
\[
  c_i=0,\,\, i\le n_1-1\mathrm{\, or\,} i\ge n_2,
  \qquad c_i\ne 0,\,\, n_1\le i\le n_2-1,
\]
and
\[
  A_1 c = 0.
\]
By multiplying a factor we can suppose that
\begin{equation}\label{eq:cc1}
\min\{|c'_i|:\,    |c'_i|>0\}  > \max\{|c_i|:\,    |c_i|>0\}.
\end{equation}
Then we have
\begin{eqnarray*}
  c_i\pm c'_i &\ne& 0, \quad i\le n_2-m-1, \\
  c_{n_2-1}\pm c'_{n_2-1}   &\ne& 0, \\
  c_i\pm c'_i &=& 0, \quad i\ge n_2.
\end{eqnarray*}
Let $f_1$ and $f_2$ be defined by (\ref{eq:fg}).
With similar arguments we get a contradiction. Hence (\ref{eq:PR:4}) is true.

{\color{blue}
Next we prove (\ref{eq:PR:2}).}
Assume on the contrary that (\ref{eq:PR:2}) is false.
Since $J_1\ge m+1$,
there is some $k_0\ge 2$ such that
$J_{k_0} \le 2k_0 + m - 2$ and $J_i  \ge 2i+m-1$ for $1\le i\le k_0-1$.
By (\ref{eq:PR:4}), $\#(E\cap (N_1+k_0-1, N_1+k_0) ) \ge 1$. Hence $J_{k_0} = 2k_0+m-2$
and $N_1+k_0-1\not\in E$.

As in the above arguments, we can split $E$ into two subsequences $E_1$ and $E_2$ which satisfy (\ref{eq:E12:1})
- (\ref{eq:E12:3}) with $n_2$ being replaced by $N_1+k_0$.
Define $A_1$ and $A_2$ by (\ref{eq:A1}) and (\ref{eq:A2}), respectively.
Then there exists some $c'\in \bbR^{N_2-N_1+m}$, for which the first $k_0$ entries are nonzero and the last $N_2-N_1-k_0+m$ entries are zeros,
such that
\[
  A_2 c' = 0.
\]
On the other hand, since $\#E_1 = k_0+m-1$, by Proposition~\ref{prop:local sampling}, $E_1$ is not a sampling sequence for $V_m|_{[N_1,N_1+k_0]}$.
Hence there is some $c\in\bbR^{N_2-N_1+m}\setminus \{0\}$, whose last $N_2-N_1-k_0$ entries are zeros, such that
\[
  A_1 c = 0.
\]
Again, we assume that (\ref{eq:cc1}) holds. With similar arguments we  get a contradiction.
Hence (\ref{eq:PR:2}) is true. Similarly we can prove (\ref{eq:PR:3}).

\textcolor{blue}{Finally, we prove (\ref{eq:PR:1})}.
Assume on the contrary that $\#E \le 2(N_2-N_1+m)-2$.
By Lemma~\ref{Lm:L1}, $\#E \ge N_2-N_1+m$.
Take some $E_1\subset E$ such that
$\#E_1 = N_2-N_1+m-1$ and
\begin{equation}\label{eq:s:1}
\#(E_1\cap (n,n+1))\ge 1, \qquad N_1\le n\le N_2-1.
\end{equation}
We see from (\ref{eq:PR:4}) that such $E_1$ exists.
Let $E_2=E\setminus E_1$.
Define $A_1$ and $A_2$ by (\ref{eq:A1}) and (\ref{eq:A2}), respectively.
Since $\rank(A_1)\le N_2-N_1+m-1$,
there is some $c\in \bbR^{N_2-N_1+m}\setminus\{0\}$ such that
\[
  A_1 c = 0.
\]
We conclude that there is not an integer $i$ such that $c_i = \ldots = c_{i+m-1} = 0$.
Otherwise, we see from (\ref{eq:s:1}) that $c_l=0$ for all $N_1-m\le l\le N_2-1$,
which contradicts with $c\ne 0$.

On the other hand, since $\#(E\setminus E_1) \le N_2-N_1+m-1$, the equation
\[
  A_2 c' = 0
\]
has a non-zero solution.
Again, we assume that (\ref{eq:cc1}) holds. With similar arguments we   get a contradiction.
This completes the proof of the necessity.
\end{proof}

To prove the sufficiency of Theorem~\ref{thm:PR}, we first present some  preliminary results.

\begin{Lemma}\label{Lm:L3a}
Let $n_1<n'_1\le n_2<n'_2$ be integers and $E\subset [n_1, n'_2]$ be a sequence consisting of distinct points such that $\#(E\cap(i,j))\ge 2(j-i)-1$
whenever  $n_1\le i<j\le n'_2$.
Suppose that  $E_1\cup E_2=E$, $E_1\cap  E_2=\emptyset $ and  that
$E_l\cap[n_l, n'_l]\ge n'_l-n_l+m$, $l=1,2$.
Then there exist integers
$i_1<i'_1=i_2<i'_2$ in $[n_1,n'_2]$ such  that 
$E_l\cap[i_l, i'_l]$  is a sampling sequence for   $V_m|_{[i_l,i'_l]}$, $l=1,2$.
\end{Lemma}

\begin{proof}
By Lemma~\ref{Lm:samp}, we may assume that
$E_l\cap[n_l, n'_l]$ is a sampling sequence for
$V_m|_{[n_l,n'_l]}$, $l=1,2$.
Note that $n'_1\le n_2$.
Let $i'_1$ be the maximum of all integers $i\in (n_1, n_2]$ such that for some $i_1\in [n_1, i)$, $E_1\cap [i_1,i'_1]$
is a sampling sequence for $V_m|_{[i_1,i'_1]}$.
And  let
$i_2$ be the minimum of all integers $i\in [i'_1, n'_2)$
 such that for some $i'_2\in (i_2,n'_2]$, $E_2\cap [i_2,i'_2]$
is a sampling sequence for $V_m|_{[i_2,i'_2]}$.
We conclude that $i'_1 = i_2$.

Assume on the contrary that $i'_1 < i_2$.
Since $\#(E\cap (i'_1,i_2))\ge 2(i_2-i'_1)-1$,
one of $\#(E_1\cap (i'_1,i_2))$
and $\#(E_2\cap (i'_1,i_2))$ must be no less than $i_2-i'_1$.
Without loss of generality, we assume that
$\#(E_2\cap (i'_1,i_2)) \ge i_2- i'_1$.
Let
\begin{equation}\label{eq:dj2}
j'_2 = \max\{ j<i_2:\, \#(E_2\cap (j,i_2)) \ge i_2-j\}.
\end{equation}
Then we have $i'_1 \le j'_2 < i_2$. For $j'_2 < j<i_2$, since
\[
  \# (E_2\cap (j, i_2)) \le i_2-j-1,
\]
we have
\[
  \# (E_2\cap (j'_2, j]) \ge j - j'_2 +1.
\]
Hence
\begin{equation}\label{eq:j2j}
\# (E_2\cap (j'_2, j)) \ge j - j'_2,\qquad j'_2 <j \le i_2.
\end{equation}
On the other hand, since
$E_2\cap [i_2,i'_2]$
is a sampling sequence for $V_m|_{[i_2,i'_2]}$, we have
\begin{equation}\label{eq:e2j:1}
\# (E_2\cap (j'_2, j)) \ge j - j'_2,\qquad j'_2 <j \le i'_2.
\end{equation}
By the choice of $i_2$,
$E_2\cap [j'_2, i'_2]$ is not a sampling sequence
for $V_m|_{[j'_2,i'_2]}$.
Since
\begin{eqnarray*}
 \#(E_2\cap (j'_2, i'_2])
   &= &  \#(E_2\cap (j'_2, i_2)) +\#(E_2\cap [i_2, i'_2])  \\
   &\ge& i'_2-j'_2+m,
\end{eqnarray*}
by Proposition~\ref{prop:local sampling}, there are two cases.

(a)\,\, There exist integers $l,l'\in [j'_2, i'_2]$ such that $l<l'$ and
$\#(E_2\cap(l, l')) \le l'-l-m-1$.

Recall that $E_2\cap[i_2,i'_2]$ is a sampling sequence for $V_m|_{[i_2,i'_2]}$.
We have $l<i_2$.
First, we assume that $l'> i_2$.
Since $\#(E_2\cap[i_2,l')) \ge l'-i_2$,
we have
\[
  \#(E_2\cap(l, i_2)) \le i_2-l -m-1.
\]
It follows from (\ref{eq:j2j}) that
\[
  \# (E_2\cap (j'_2, l)) \ge  l-j'_2+m.
\]
By Lemma~\ref{Lm:samp}, there exist integers
$l_2, l'_2\in [j'_2, l]$  such that $l_2<l'_2$ and   $E_2\cap [l_2, l'_2]$ is sampling sequence for $V_m|_{[l_2, l'_2]}$,
which contradicts with the choice of $i_2$.

Next we consider the case of $l'\le i_2$.
We see from (\ref{eq:dj2}) that
\begin{eqnarray*}
\#(E_2\cap (l,i_2))
&=& \#(E_2\cap (l,l')) + \#(E_2\cap [l',i_2)) \\
&\le& l'-l-m-1 +  i_2 - l' = i_2-l-m-1.
\end{eqnarray*}
With similar arguments we get a contradiction.

(b).\,\,  There exists some inter $l\in [j'_2,  i'_2)$ such that
$\#(E_2\cap(l, i'_2]) \le i'_2-l-1$.

In this case, we have $l<i_2$
and
\begin{eqnarray*}
\#(E\cap(l,i_2)) &=& \#(E\cap(l,i'_2]) - \#(E\cap[i_2,i'_2]) \\
&\le& i_2-l-m-1.
\end{eqnarray*}
Again, with similar arguments we get a contradiction. This completes the proof.
\end{proof}

The following is a simple application of Lemma~\ref{Lm:L3a}.

\begin{Lemma}\label{Lm:L3}
Let $E\subset [N_1, N_2]$ be a sequence consisting of distinct points which satisfies (\ref{eq:PR:1}) - (\ref{eq:PR:4}).
Suppose that  $E_1\cup E_2=E$, $E_1\cap  E_2=\emptyset $, $\#E_1\le  \# E_2$, and that  $E_2$  is not a sampling sequence for
$\VN$.
Then there exist integers
 $i_1, i'_1, i_2, i'_2\in [N_1, N_2]$ such that
 $i_1<i'_1$, $i_2<i'_2$,  $[i_1, i'_1]$ and $[i_2, i'_2]$ have and only have one common point,
 and
$E_l\cap[i_l, i'_l]$  is a sampling sequence for   $V_m|_{[i_l,i'_l]}$, $l=1,2$.
\end{Lemma}

\begin{proof}
Since $\# E_2 \ge \#E_1$, we see from (\ref{eq:PR:1}) that
\begin{equation}\label{eq:aa:e1}
\#E_2 \ge N_2-N_1+m.
\end{equation}
Since
$E_2$  is not a sampling sequence for
$\VN$, by Proposition~\ref{prop:local sampling}, there are three cases.

(i) There is some integer $k\in [1, N_2-N_1]$ such that $\#(E_2\cap[N_1, N_1+k)) \le k-1$.

In this case, we see from (\ref{eq:PR:2}) that
\begin{eqnarray*}
 \#(E_1\cap[N_1, N_1+k)) &=& \#(E\cap[N_1, N_1+k)) - \#(E_2\cap[N_1, N_1+k)) \\
   &\ge& k+m.
\end{eqnarray*}
By Lemma~\ref{Lm:samp}, there exist integers $n_1, n'_1$ such that $N_1\le n_1<n'_1\le N_1+k$ and
$E_1\cap [n_1,n'_1]$ is a sampling sequence for $V_m|_{[n_1,n'_1]}$.

On the other hand, since
\begin{eqnarray*}
 \#(E_2\cap[N_1+k, N_2]) &=& \#(E_2\cap[N_1, N_2]) - \#(E_2\cap[N_1, N_1+k)) \\
   &\ge& N_2 - N_1-k+m+1,
\end{eqnarray*}
Using Lemma~\ref{Lm:samp} again, we get some integers $n_2, n'_2$ with
$N_1+k\le n_2<n'_2\le N_2$ such that
$E_2\cap [n_2,n'_2]$ is a sampling sequence for $V_m|_{[n_2,n'_2]}$.
By Lemma~\ref{Lm:L3a}, we get the conclusion as desired.

(ii) There is some integer $k\in [1, N_2-N_1]$  such that $\#(E_2\cap(N_2-k, N_2]) \le k-1$.

Similarly to the first case we can prove the conclusion.

(iii) There exist integers $k_1, k_2\in[N_1, N_2] $ such that $k_1<k_2$  and $\#(E_2\cap(k_1, k_2)) \le k_2-k_1-m-1$.

In this case, we have
\[
  \#(E_1\cap (k_1,k_2)) \ge k_2-k_1+m.
\]
Hence there exist integers $n_1, n'_1\in [k_1, k_2]$ such that $n_1 < n'_1$ and
$E_1\cap [n_1,n'_1]$ is a sampling sequence for $V_m|_{[n_1,n'_1]}$.

On the other hand,
If $\#(E_2\cap [N_1, k_1]) \le  k_1-N_1+m-1$ and
$\#(E_2\cap [k_2, N_2]) \le N_2-k_2 +m-1$, then
\begin{eqnarray*}
\#(E_2\cap[N_1, N_2])
\le N_2-N_1+m-3,
\end{eqnarray*}
which contradicts with (\ref{eq:aa:e1}).
Hence either $\# (E_2\cap[N_1, k_1])\ge k_1-N_1+m$
or $\#(E_2\cap[k_2, N_2]) \ge N_2-k_2+m$.
Consequently, we can find some integers $n_2<n'_2$ such that
$[n_2,n'_2]\subset [N_1, k_1]$
or $[n_2,n'_2]\subset [k_2, N_2]$
and $E_2\cap [n_2, n'_2]$ is a sampling sequence for $V_m|_{[n_2,n'_2]}$.
Again, the conclusion follows from Lemma~\ref{Lm:L3a}.
\end{proof}


We finish the proof of Theorem~\ref{thm:PR} by presenting a slightly stronger result,
of which the sufficiency of Theorem~\ref{thm:PR} is a consequence.

\begin{Lemma}\label{Lm:L4}
Suppose that $E=\{x_i:\, 1\le i\le N\}$ satisfies (\ref{eq:PR:1}) - (\ref{eq:PR:4}).
Let $f_i = \sum_{n=N_1-m}^{N_2-1} c^{(i)}_n \varphi_m(\cdot-n)\in V_m$ ($i=1,2$)  be such that
\begin{equation}\label{eq:f12i}
|f_1(x_k)| = |f_2(x_k)|, \qquad 1\le k\le N.
\end{equation}
Then $|f_1(x)|=|f_2(x)|$ on $[N_1, N_2]$. Moreover,
if $f_1\ne \pm f_2$,  then
there exist integers $n_1, n_2\in [N_1, N_2]$  such that $n_2-n_1\ge m+1$,
\begin{equation}\label{eq:cn}
c^{(i)}_n=0, \quad  n_1<n<n_2,\,\,  i=1,2,
\end{equation}
and there exist $j^{(1)}_1, j^{(2)}_1 \le n_1$ and
$j^{(1)}_2, j^{(2)}_2 \ge n_2$ satisfying that
none of $c^{(1)}_{j_1}$, $c^{(1)}_{j_2}$, $c^{(2)}_{j_1}$ and $c^{(2)}_{j_2}$ is zero.
\end{Lemma}

\begin{proof}
Suppose that $f_1\pm f_2\ne 0$.
Split $E$ into two subsequences $E_1$ and $E_2$ such that
\[
  E_1 = \{x_i:\, f_1(x_i) = -f_2(x_i)\}\quad \mathrm{and}\quad  E_2 = \{x_i:\, f_1(x_i) = f_2(x_i)\}.
\]
Since $\#E \ge 2(N_2-N_1+m)-1$, without loss of generality, we assume that
$\# E_2 \ge N_2-N_1+m$.

Since $f_1-f_2\ne 0$, $E_2$ is not a sampling sequence for $\VN$.
We see from   Proposition~\ref{prop:local sampling} that $N_2-N_1\ge 2$.
We prove the conclusion with induction on $N_2-N_1$.

First, we consider the case of $N_2-N_1=2$.
Since $E_2$ is not a sampling sequence for $\VN$, we see from Proposition~\ref{prop:local sampling}
that either $E_2\cap [N_1, N_1+1)=\emptyset$ or $E_2\cap (N_1+1, N_2]=\emptyset$.
Without loss of generality, assume that $E_2\cap [N_1, N_1+1)=\emptyset$.

Since  $f_1 \pm f_2 \ne 0$, there are some $c, c'\in \bbR^{2+m}\setminus\{0\}$ such that
\begin{equation} \label{eq:f12}
\left\{
\begin{aligned}
f_1(x) + f_2(x) &= \sum_{n=N_1-m}^{N_2-1} c_n \varphi_m(x-n), \\
f_1(x) - f_2(x) &= \sum_{n=N_1-m}^{N_2-1} c'_n \varphi_m(x-n).
\end{aligned}\right.
\end{equation}
Since  $\#(E_2\cap[N_1+1,N_2]) = \#(E_2\cap[N_1,N_2]) \ge 2+m$,
$E_2\cap[N_1+1, N_2]$ is a sampling sequence for $V_m|_{[N_1+1,N_2]}$.
We see from $f_1(x_i) = f_2(x_i)$ for $x_i\in E_2$ that
$f_1(x)=f_2(x)$ on $[N_1+1,N_2]$ and
\[
  c'_n=0,\qquad N_2-m-1=N_1-m+1 \le n\le N_1+1 = N_2-1.
\]
But $f_1\ne  f_2$. Hence $c'_{N_1-m}\ne 0$.

On the other hand, since    $\#(E_1\cap[N_1 , N_1+1)) = \#(E\cap[N_1,N_1+1)) \ge 1+m$,
$E_1\cap [N_1, N_1+1]$ is a sampling sequence for $V_m|_{[N_1,N_1+1]}$.
Now we see from  $f_1(x_i) = -f_2(x_i)$ for $x_i\in E_1$ that
$f_1(x)=-f_2(x)$ on $[N_1,N_1+1]$ and
\[
  c_n=0,\qquad N_1-m \le n\le N_1.
\]
But $f_1\ne -f_2$. Hence $c_{N_1+1}\ne 0$.
Therefore,
\[
   c_n=c'_n = 0, \qquad  N_1-m+1 \le n\le N_1
\]
and
\[
   c_n\pm c'_n \ne 0, \qquad  n=N_1-m \, \mathrm{or}\, N_2-1.
\]
By Lemma~\ref{Lm:sep}, both $f_1$ and $f_2$ are separable.
Consequently, the conclusion is true for $N_2-N_1=2$.

Now suppose that for some $N_0\ge 2$ and any integers $N_1< N_2$ with $N_2-N_1\le N_0$ the conclusion is true.
Let us consider the case of $N_2-N_1=N_0+1$.

Since $E_2$ is not a local sampling sequence for $\VN$,
by Lemma~\ref{Lm:L3}, there exist integers
$i_1, i'_1, i_2, i'_2\in [N_1, N_2]$ such that
$i_1<i'_1$, $i_2<i'_2$,  $[i_1, i'_1]$ and $[i_2, i'_2]$ have and only have one common point,
 and
$E_1\cap[i_1, i'_1]$  and
$E_2\cap[i_2, i'_2]$
are sampling sequences for   $V_m|_{[i_1,i'_1]}$
and $V_m|_{[i_2,i'_2]}$, respectively.
Without loss of generality, we assume that $i'_1=i_2$.

Take some $c,c'\in\bbR^{N_2-N_1+m}$ such that (\ref{eq:f12}) holds.
Since $E_1\cap [i_1, i'_1]$ is a sampling sequence for $[i_1,i'_1]$ and
$f_1(x_i) = -f_2(x_i)$ for $x_i\in E_1$, we have
\begin{equation}\label{eq:cni}
c_n = 0,\qquad i_1-m \le n\le i'_1-1.
\end{equation}
Similarly we get
\[
   c'_n = 0,\qquad i_2-m \le n\le i'_2-1.
\]
Hence
\[
  c_n=c'_n = 0,\qquad i'_1-m \le n\le i'_1-1 = i_2-1.
\]
Set $n_1 = i_2-m-1$ and $n_2=i_2$. Then we get (\ref{eq:cn}).

Since $f_1 \pm f_2\ne 0$, neither $f_1$ nor $f_2$ is zero.
Otherwise, we see from $|f_1(x_i)|=|f_2(x_i)|$ that both are zeros.
Hence there exist some $k_1, k_2$ such that $c_{k_1}+c'_{k_1}\ne 0$ and $c_{k_2} - c'_{k_2}\ne 0$.
Without loss of generality, assume that $k_1\le k_2$.

First, we assume that $k_1\le i_2-1\le k_2$.
Then we have $k_1 \le i'_1-m-1$ and $k_2 \ge i'_1=i_2$.

Note that $f_1 = \sum_{n=N_1-m}^{N_2-1} \frac{c_n+c'_n}{2} \varphi_m(\cdot-n)$.
We have  $f_1|_{[N_1,i'_1]}\ne 0$.
Since $E\cap [N_1,i'_1]$ is a sampling sequence for $V_m|_{[N_1,i'_1]}$,
we see from (\ref{eq:f12i}) that $f_2|_{[N_1,i'_1]}\ne 0$.
Similarly we can prove that $f_i|_{[i_2,N_2]} \ne 0$, $i=1,2$.
Hence we can find integers $j^{(i)}_l$, $1\le i,l\le 2$, as desired.

Next we assume that for any $k_1,k_2$ with
$c_{k_1}+c'_{k_1}\ne 0$ and $c_{k_2} - c'_{k_2}\ne 0$,
$k_1$ and $k_2$ lie on the same side of $i_2-1$.
Without loss of generality, we assume that $k_1, k_2\le i_2-1$. Then we have
\[
  c_n=c'_n = 0,\qquad n\ge i_2-m.
\]
Hence
\[
  \left\{
\begin{aligned}
f_1(x) + f_2(x) &= \sum_{n=N_1-m}^{i_2-m-1} c_n \varphi_m(x-n), \\
f_1(x) - f_2(x) &= \sum_{n=N_1-m}^{i_2-m-1} c'_n \varphi_m(x-n).
\end{aligned}\right.
\]
Take $m+1$ points $y_1, \ldots, y_{m+1}$ in $(i_2-1,i_2)\setminus E$
and set $\tilde E = (E\cap[N_1, i_2])\cup \{y_i:\,1\le i\le {m+1}\}$.
Then $\tilde E$ meets (\ref{eq:PR:1}) - (\ref{eq:PR:4}) if
$(E, N_2)$ is replaced by $(\tilde E, i_2)$.

Let $\tilde E_1 = (E_1 \cap [N_1, i_2])\cup \{y_i:\,1\le i\le {m+1}\}$
and $\tilde E_2 = E_2\cap [N_1,i_2]$.
Then we have $\tilde E = \tilde E_1 \cup \tilde E_2$ and $f_1(x_i) = f_2(x_i)$ for $x_i\in \tilde E_2$.

On the other hand, since $E_1\cap [i_1, i'_1]$ is a sampling sequence for $V_m|_{[i_1,i'_1]}$
and $(f_1+f_2)|_{E_1} = 0$, we have $(f_1+f_2)|_{[i_1,i'_1]} = 0$.
Hence $f_1(y_i) + f_2(y_i) = 0$, $1\le i\le m+1$.
Therefore $(f_1+f_2)|_{\tilde E_1}=0$.
Since $i_2-N_1\le  N_2-N_1-1$, we see from the inductive assumption that
we can find integers $j^{(i)}_l$, $1\le i,l\le 2$, as desired.

Finally, we prove that $|f_1(x)|=|f_2(x)|$. We see from the above arguments that there exist integers $i_1<i'_1=i_2<i'_2$ such that
$E_l\cap [i_l,i'_l]$ is a sampling sequence for $V_m|_{[i_l,i'_l]}$, $l=1,2$.
Hence $f_1 = -f_2$ on $[i_1, i'_1]$
and  $f_1 = f_2$ on $[i_2, i'_2]$.
Take some $F_l\subset [i_l, i'_l]$ such that $\#F_l = 2m+3$.
Then $F_1\cup E\cap[N_1, i'_1]$ and $F_2\cup E\cap[i_2, N_2]$ are phaseless sampling sequences for $V_m|_{[N_1,i'_1]}$
and $V_m|_{[i_2,N_2]}$, respectively.
Since $i'_1-N_1 <N_2-N_1$ and $N_2-i_2<N_2-N_1$, we see from the inductive assumption that
$|f_1(x)|=|f_2(x)|$ on $[N_1, N_2]$.

By induction, the conclusion is true for any $N_1<N_2$. This completes the proof.
\end{proof}

\section{Global Phaseless Sampling in Spline Spaces}

In this section, we give a proof for Theorem~\ref{thm:global}.
First, we present some equivalent characterization for phaseless sampling sequences,
for which we leave the proof to interested readers.

\begin{Lemma} \label{Lm:s4:L1}
Suppose that $m\ge 2$. For a sequence $E$ satisfying (P1) in Theorem~\ref{thm:global}, (P2) is equivalent to any one of the followings,

\begin{enumerate}
\item For any $n_0\in\bbZ$,
\[
 \! \lim_{n\rightarrow \infty} (\#(E\cap [n_0, n]) - 2(n-n_0))
  \!=\!
  \lim_{n\rightarrow -\infty} (\#(E\cap [n, n_0]) - 2(n_0-n)) = \infty.
\]

\item For any $n_0\in\bbZ$,
\[
   \begin{aligned}
  &\sup_{n>n_0} \Big(\#(E\cap [n_0,n]) - 2(n-n_0)\Big)  \\
   &\hskip 10mm = \sup_{n<n_0} \Big(\#(E\cap [n,n_0]) - 2(n_0-n)\Big)=\infty.
 \end{aligned}
\]
\end{enumerate}

\end{Lemma}

\begin{proof}[Proof of Theorem~\ref{thm:global}]
\textcolor{blue}{Necessity}.\,\,  First, we show that $\#(E\cap(n,n+1))\ge 1$ for any $n\in\bbZ$.
Assume on the contrary that $\#(E\cap(n,n+1)) = 0$ for some $n\in\bbZ$.
Let
\begin{eqnarray*}
  f_1 &=& \sum_{k\ge n} \varphi_m(\cdot-k) + \sum_{k\le n-m} \varphi_m(\cdot-k), \\
  f_2 &=& \sum_{k\ge n} \varphi_m(\cdot-k) - \sum_{k\le n-m} \varphi_m(\cdot-k), \\
\end{eqnarray*}
Then we have
\begin{eqnarray*}
  f_1(x) &=& -f_2(x),\qquad x \le n, \\
  f_1(x) &=& f_2(x),\qquad x \ge n+1.
\end{eqnarray*}
Hence $|f_1(x)| = |f_2(x)|$ for $x\in E$. But $f_1 \pm f_2\ne 0$ and neither $f_1$ nor $f_2$ is separable,
which contradicts with the hypothesis.

Now assume that $\#(E\cap(i_1, i_2)) \le 2(i_2-i_1)-2$ for some $i_1<i_2-1$.
Let
\begin{eqnarray*}
  n_1 &=& \max\{n<i_2:\, \#(E\cap(n, i_2)) \le 2(i_2-n)-2 \}, \\
  n_2 &=& \min\{n>n_1:\, \#(E\cap(n_1, n)) \le 2(n-n_1)-2 \}.
\end{eqnarray*}
Then we have $i_1\le n_1<n_2\le i_2$ and (\ref{eq:n1n2:1}) - (\ref{eq:n1n2:5}) hold.
Hence there is some $E_2\subset E$ such that
\begin{eqnarray*}
 && E\cap ((-\infty, n_1]\cup[n_2,\infty)) \subset E_2, \\
 && E_2\cap (n,n+1) = 1, \qquad n_1\le n\le n_2-2, \\
 && E_2\cap (n_1,n_2) = n_2-n_1-1, \qquad n_1\le n\le n_2-2.
\end{eqnarray*}
Let $E_1 = E\setminus E_2$ and
\begin{eqnarray*}
A_1 &=& [\varphi_m(x_i-n)]_{x_i\in E_1, n_1\le n\le n_2-1}, \\
A_2 &=& [\varphi_m(x_i-n)]_{x_i\in E_2\cap(n_1,n_2), n_1-m\le n\le n_2-m-1}.
\end{eqnarray*}
Since $\#E_1 = n_2-n_1-1$, we see from Proposition~\ref{prop:Sch} that there is some $c\in \bbR^{n_2-n_1}$,
none of whose entries is zero, such that
\[
  A_1 c = 0.
\]
Similarly, there is some $c'\in \bbR^{n_2-n_1}$,
none of whose entries is zero, such that
\[
  A_2 c' = 0.
\]
Again, by multiplying a factor, we can assume that
$c_i\pm c'_i\ne 0$, $n_1\le i\le n_2-m-1$.
Let
\begin{eqnarray*}
  f_1 &=& \sum_{n=n_1}^{n_2-1} c_n \varphi_m(\cdot-n) + \sum_{n=n_1-m}^{n_2-m-1} c'_n \varphi_m(\cdot-n), \\
  f_2 &=& \sum_{n=n_1}^{n_2-1} c_n \varphi_m(\cdot-n) - \sum_{n=n_1-m}^{n_2-m-1} c'_n \varphi_m(\cdot-n).
\end{eqnarray*}
If $n_1\le n_2-m-1$, then both $f_1$ and $f_2$ are nonseparable.
If $n_1>n_2-m-1$, then $n_1 - (n_2-m-1) \le m-1$. Again,  both $f_1$ and $f_2$ are nonseparable.
Moreover,
\[
f_1(x) = \begin{cases}
-f_2(x), &\text{if $x\le n_1$ or $x\in E_1=E_1\cap(n_1,n_2)$}\\
f_2(x),  &\text{if $x\ge n_2$ or $x\in E_2\cap(n_1,n_2)$}.
\end{cases}
\]
Since $f_1\pm f_2\ne 0$, it is impossible to recover $f_1$ or $f_2$, which contradicts with the hypothesis.
Hence (P1) is true.

\textcolor{blue}{Next we prove (P2) for $m\ge 2$}.
Assume that for some $n_0$ and any $n_2>n_1\ge n_0$,
$\#(E\cap [n_1, n_2]) \le 2(n_2-n_1+m)-2$. Then we have
\[
  K:= \sup_{n>n_0} \#(E\cap [n_0,n]) - 2(n-n_0)<\infty.
\]
Otherwise, there is some $n>n_0$ such that $\#(E\cap [n_0,n]) - 2(n-n_0)>  4m$.
Hence either $\#(E\cap [n_0,n_0+1]) \ge 2m+1$ or
$\#(E\cap [n_0+1,n]) \ge 2(n-n_0-1)+2m-1$, which contradicts with the assumption.

It follows that  there exists some $n_2>n_0$ such that
\[
  K = \#(E\cap [n_0,n_2]) - 2(n_2-n_0).
\]
For $n>n_2$, we have
\[
  \#(E\cap (n_2, n]) \le 2(n-n_2).
\]
Moreover,
\begin{equation}\label{eq:s5:e1}
\#(E\cap(n, n+1]) \le 3,\qquad n\ge n_2.
\end{equation}
Otherwise, we have
\[
  \#(E\cap(n_2,n+1]) = \#(E\cap(n_2,n]) + \#(E\cap(n,n+1]) \ge 2(n-n_2)+3.
\]
Hence
\[
  \#(E\cap[n_0,n+1]) - 2(n+1-n_0) \ge K+1,
\]
which is impossible.

Let  $a_i = \#(E\cap (i-1,i])$, $i\ge n_2+1$. Then we have $1\le a_i\le 3$.
There are three cases.

(i).\,\, There are infinitely many $i> n_2$ such that $a_i=3$.

Suppose that $a_{i_k}=3$ for $k\ge 1$ and $a_i\le 2$ for $i\ge n_2+1$ and $i\not\in\{i_k:\,k\ge 1\}$.
Set $i_0 = n_2$.
We conclude that
$i_k - i_{k-1} \ge 2$ for $k\ge 1$.

In fact, if $i_1 = n_2+1$, then we have $\#(E\cap [n_0, i_1]) - 2(i_1-n_0)=K+1$,  which is impossible.
On the other hand,  if $i_k = i_{k-1}+1$ for some $k\ge 2$, then
\begin{eqnarray*}
\#(E\cap [n_0, i_k]) &=& \#(E\cap[n_0,n_2]) + \#(E\cap(n_2,i_{k-1}-1]) \\
&&\qquad + \#(E\cap(i_{k-1}-1,i_k]) \\
              &\ge& K + 2(n_2-n_0)+ 2(i_{k-1} - 1-n_2) -1 + 6 \\
              &=& K + 2(i_k -n_0)+1.
\end{eqnarray*}
Again, we get a contradiction.

Observe that
\begin{eqnarray*}
  \#(E\cap (i_{k-1},i_k])
  &=& \#(E\cap(i_{k-1},i_k-1]) + \#(E\cap(i_k-1,i_k]) \\
  &\ge& 2(i_k-i_{k-1}).
\end{eqnarray*}
It is easy to check  by induction that for $k\ge 1$,
\begin{equation}\label{eq:s5:e3}
  \#(E\cap (i_{k-1},i_k]) = 2(i_k-i_{k-1}).
\end{equation}

(ii).\,\, There are only finitely many $i> n_2$ such that $a_i=3$.

Suppose that $a_{i_1}$, $\ldots$,  $a_{i_r}=3$. Let $i_0=n_2$ and $i_k = i_r+2(k-r)$ for $k>r$.
Similarly we can show that (\ref{eq:s5:e3}) is true for $1\le k\le r$.  Hence
\begin{equation}\label{eq:s5:e4}
 2(i_k-i_{k-1})-1 \le \#(E\cap (i_{k-1},i_k]) \le  2(i_k-i_{k-1}),\qquad k\ge 1.
\end{equation}

(iii).\,\, $a_i \le 2$ for $i > n_2$.

In this case, set $i_k = n_2 + 2k$, $k\ge 0$. Then (\ref{eq:s5:e4}) is true.

In all three cases, we get $n_2=i_0<i_1<\ldots<i_k<\ldots$ such that $i_k-i_{k-1} \ge 2$ and (\ref{eq:s5:e4}) is true.

Take some $E_2\subset E$ such that
\begin{eqnarray}
 && \#(E_2\cap (n,n+1)) \ge 1, \qquad n\in\bbZ,   \label{eq:s5:e6} \\
 && \#(E_2\cap (i_0,i_1]) = \#(E_2\cap [i_0,i_1]) = i_1-i_0+1, \nonumber \\
 && \#(E_2\cap (i_{k-1},i_k]) = i_k-i_{k-1},\qquad k\ge 2,\nonumber \\
 && \#(E_2\cap [n,i_0)) = i_0-n,\qquad n<i_0,\nonumber \\
 && E\cap (i_0, i_0+1] \subset E_2. \nonumber
\end{eqnarray}
Let $E_1= E\setminus E_2$. Then
\begin{eqnarray}
 && \#(E_1\cap(i_0,i_0+1]) = 0,  \label{eq:s5:e6a}  \\
 && \#(E_1\cap(i_0,i_1]) \le  i_1-i_0-1,  \label{eq:s5:e6b} \\
 && \#(E_1\cap(i_{k-1},i_k]) \le  i_k-i_{k-1},\qquad k\ge 2.  \label{eq:s5:e6c}
\end{eqnarray}

Since $\#(E_2\cap [i_0,i_1]) = i_1-i_0+1 < i_1-i_0+m$,
there exist some $c'_{i_0-m}$, $\ldots$, $c'_{i_1-1}\in\bbR$, not all of which are zeros, such that
\[
  h_2(x_i) = \sum_{n=i_0-m}^{i_1-1} c'_n \varphi_m(x_i-n) = 0,\qquad x_i\in E_2\cap    [i_0, i_1].
\]

Consider the linear system
\[
  h_2(x_i) = \sum_{n=i_0-m}^{i_2-1} c'_n \varphi_m(x_i-n) = 0,\qquad x_i\in E_2\cap (i_1, i_2],
\]
where $c'_n$, $i_1\le n\le i_2-1$ are considered as unknowns.
Denote  $E_2\cap (i_1, i_2] = \{y_i:\, 1\le i\le i_2-i_1\}$. Then the above system can be rewritten as
\[
  \begin{pmatrix}
   * & 0    & \ldots & 0 \\
   * & *    & \ldots & 0 \\
     &  \ldots  & \ldots & 0 \\
   ? & ?   & \ldots & *
\end{pmatrix}
 \begin{pmatrix}
   c'_{i_1} \\
   \vdots \\
   c'_{i_2-1}
 \end{pmatrix} =
  \begin{pmatrix}
   b_1 \\
   \vdots \\
   b_{i_2-i_1}
 \end{pmatrix}
\]
where $*$ stands for non-zero entries and $b_i = -\sum_{n=i_0-m}^{i_1-1} c'_n\varphi_m(y_i-n)$, $1\le i\le i_1-n_2$.
Hence there is a unique solution to the linear system. Consequently, for $x_i\in E_2\cap (i_1, i_2]$,
\[
  h_2(x_i) = \sum_{n=i_0-m}^{i_2-1} c'_n \varphi_m(x_i-n) = 0.
\]
Note that the above identity is also true for $x_i\in E_2\cap [i_0,i_1]$ since $\varphi_m(x_i-n)=0$ for $x_i\le i_1 \le n$.
By induction, it is easy to see that there exists a sequence of real numbers  $\{c'_n:\,n\in\bbZ\}$ such that
\begin{equation}\label{eq:h2}
h_2(x_i) = \sum_{n\in\bbZ} c'_n \varphi_m(x_i-n) = 0,\qquad x_i\in E_2.
\end{equation}

On the other hand,
by (\ref{eq:s5:e6a}) and (\ref{eq:s5:e6b}),
there exist some $c_{i_0}$, $\ldots$, $c_{i_1-1}\in\bbR$, not all of which are zeros, such that
\[
  h_1(x_i) = \sum_{n=i_0}^{i_1-1} c_n \varphi_m(x_i-n) = 0,\qquad x_i\in E_1\cap (i_0,i_1].
\]
Put $c_n=0$ for $n<i_0$. We get
\[
  h_1(x_i) = \sum_{n\le i_1-1} c_n \varphi_m(x_i-n) = 0,\qquad x_i\in E_1\cap (-\infty,i_1].
\]
Since $\#(E_1\cap(i_{k-1},i_k]) \le i_k - i_{k-1}$, $k\ge 2$,
similarly we can find some $c_n\in\bbR$ for $n\ge i_1$, which might not be unique, such that
\[
  h_1(x_i) = \sum_{n\in\bbZ} c_n \varphi_m(x_i-n) = 0,\qquad x_i\in E_1.
\]

For any $n\in\bbZ$ with $c'_n\ne 0$, there are at most two numbers $r$ such that $c_n +r c'_n = 0$ or $c_n -r c'_n = 0$.
Hence there is some $r\in [1,2]$ such  that $c_n\pm rc'_n\ne 0$ if $c'_n\ne 0$.
Set
\begin{eqnarray*}
f_1 &=& \sum_{n\in\bbZ}  \frac{c_n+rc'_n}{2} \varphi_m(\cdot-n), \\
f_2 &=& \sum_{n\in\bbZ}  \frac{c_n-rc'_n}{2} \varphi_m(\cdot-n).
\end{eqnarray*}
By (\ref{eq:s5:e6}),  there is not an integer  $n$ such that $c'_n = \ldots = c'_{n+m-1} = 0$. Otherwise, $c'_n=0$ for all $n\in\bbZ$.
The same is true for $c_n \pm rc'_n$. Hence both $f_1$ and $f_2$ are nonseparable. Since $f_1\pm f_2\ne 0$ and
$|f_1(x_i)| = |f_2(x_i)|$ for $x_i\in E$, we can not recover $f_1$ up to a sign, which contradicts with the hypothesis.
Similarly we can show that for any integer $n_0$, there exist integers  $i_1<i_2\le n_0$ such that
 $\#(E\cap [i_1, i_2]) \ge 2(i_2-i_1+m)-1$.
Hence (P2) is true.

\textcolor{blue}{It remains to prove (P2') for $m=1$.}
By (P1), there is some $E_2 = \{x_i:\, i\in\bbZ\}\subset E$ such that $x_i\in (i,i+1)$.
Let $E_1=E\setminus E_2$.
Then $E_1\ne \emptyset$.

Take some $c'_{-1}, c'_0\in\bbR\setminus\{0\}$ such that
\[
   c'_{-1} \varphi_1(x_0+1) + c'_0 \varphi_1(x_0) = 0.
\]
By induction, for any $n\in\bbZ$, we can find $c'_n\in\bbR$ successively such that
\[
  c'_{n-1} \varphi_1(x_n-n+1) + c'_n \varphi_1(x_n-n) = 0.
\]
It is easy to see that $c'_n\ne 0$, $n\in\bbZ$.
Set $h_2(x) =   \sum_{n\in\bbZ} c'_n \varphi_m(x-n)$. Then $h_2|_{E_2}=0$.

Denote $b_n = \#(E\cap [n-1,n])$, $n\in\bbZ$.
First, we show that there exists some $n\in\bbZ$ such that $b_n\ge 3$.

Assume on the contrary that $b_n\le 2$, $n\in\bbZ$.
Then $E_1\cap [n,n+1] \le 1$ for any $n\in\bbZ$.
Since $\#(E\cap(0,3))\ge 5$ and $\#(E_2\cap(0,3))=3$,
we have $\#(E_1\cap(0,3))\ge 2$. Moreover, since $b_n\le 2$, $1$ or $2$ is not in $E_1$.
Consequently, there is some $y_0\in E_1\cap(0,3)\setminus\bbZ$.
Since $\#(E_1\cap [n,n+1]) \le 1$, it is easy to see that there  is some $h_1\in V_1$ such that $h_1\ne0$ and $h_1|_{E_1} = 0$.
Since $h_2$ is nonseparable and $h_1, h_2\ne  0$, similarly to the proof of (P2) we
get a contradiction.

Take some $n_0\in \bbZ$ such that $b_{n_0}\ge 3$.
Denote $\{n:\, b_n\ge 3\}$ by $\{n_k:\, k_1\le k\le k_2\}$.

Suppose that $k_2<\infty$. Then $b_n\le 2$ for any $n >  n_{k_2}$.
If there is some $i_0\in E\cap [n_{k_2},\infty)\cap \bbZ$,   it follows from $b_{i_0+1}\le 2$ that
$\#(E\cap (i_0, i_0+1])=1$.
Hence $E_1\cap (i_0, i_0+1]=\emptyset $.
Since $\#(E\cap(i_0,i_0+2)) \ge 3$, there is some $y_0\in (i_0+1, i_0+2)\cap E_1$.
Hence there exist constants $c_n$, $n\ge i_0$, not all of which are zeros, such that
\[
  h_1(x) := \sum_{n\ge i_0} c_n \varphi_1(x-n) = 0
\]
holds for $x\in E_1\cap (i_0+1,\infty)$ and  therefore for all $x\in E_1$.
Similarly to the above  we
get a contradiction.
Hence $E\cap [n_{k_2},\infty )\cap \bbZ =  \emptyset$.

If $b_n = 1$ for some $n > n_{k_2}$, with similar arguments we get a contradiction.
Hence
\[
  \#(E\cap (n,n+1)) = 2, \qquad n\ge n_{k_2}.
\]
Similar arguments show that
$\#(E\cap (n-1,n) = 2$ for $n\le n_{k_1}$.

\textcolor{blue}{Sufficiency}.\,\,
Let $f_1, f_2\in V_m$ be such that $|f_1(x_i)|=|f_2(x_i)|$ for $x_i\in E$.
Assume that $f_1\pm f_2\ne 0$.
Then neither $f_1$ nor $f_2$ is equal to zero.
In fact, we see from (P1) that $E\cap [n_1, n_2]$ is a sampling sequence for $V_m|_{[n_1, n_2]}$ whenever $n_2-n_1\ge m+1$.
Since $|f_1(x)|=|f_2(x)|$ on $E$, if $f_1=0$, then we have $f_2=0$. Therefore, $f_1\pm f_2=0$, which is a contradiction.

First, we consider the case of $m\ge 2$.
Suppose that a sequence $E$  satisfies (P1) and (P2).
Define
\[
  E_1 = \{x_i:\, f_1(x_i) = -f_2(x_i)\}\quad \mathrm{and}\quad  E_2 = \{x_i:\, f_1(x_i) = f_2(x_i)\}.
\]

By (P2), there exists integers $n_i$ and $n'_i$, $i\in\bbZ$, such that
\[
   n_i<n'_i\le n_{i+1}<n'_{i+1},\qquad i\in\bbZ
\]
and
\[
 \#(E\cap [n_i, n'_i]) \ge 2(n'_i-n_i+m)-1.
\]
Hence  for any $i\in\bbZ$, either
$\#(E_1\cap [n_i, n'_i]) \ge n'_i-n_i+m $
or $\#(E_2\cap [n_i, n'_i]) \ge n'_i-n_i+m $.
There are two case.

(i)\,\, There exist  some $i\ne j$ such that
$\#(E_1\cap [n_i, n'_i]) \ge n'_i-n_i+m$ and $\#(E_2\cap [n_j, n'_j]) \ge n'_j-n_j+m$.

Without loss of generality, we assume that $i<j$.
We see from Lemma~\ref{Lm:L3a} that
there exist integers
$i_1< i'_1=i_2< i'_2$ such that
$E_l\cap[i_l, i'_l]$  is a sampling sequence for   $V_m|_{[i_l,i'_l]}$, $l=1,2$.

(ii)\,\,For  $l=1$ or $2$, $\#(E_l\cap [n_i, n'_i]) \ge n'_i-n_i+m$ for all $i\in\bbZ$.

Without loss of generality, assume that
$\#(E_2\cap [n_i, n'_i]) \ge n'_i-n_i+m$ for all $i\in\bbZ$.
By Lemma~\ref{Lm:samp},
there exists $[\tilde n_i, \tilde n'_i]\subset [n_i, n'_i] $
such that $E_2\cap [\tilde n_i, \tilde n'_i]$ is a sampling sequence
for $V_m|_{[\tilde n_i, \tilde n'_i]}$, $i\in\bbZ$.

If $E_2\cap [\tilde n_i, \tilde n'_j]$ is a sampling sequence for $V_m|_{[\tilde n_i, \tilde n'_j]}$ for any $i<j$, then we have
$f_1-f_2=0$ on $[\tilde n_i, \tilde n'_j ]$ for any $i<j$. Hence $f_1=f_2$, which contradicts with the assumption.
Consequently,  there exist some $i<j$ such that
$E_2\cap [\tilde n_i, \tilde n'_j]$ is not a sampling sequence for $V_m|_{[\tilde n_i,  \tilde n'_j]}$.
By Proposition~\ref{prop:local sampling}, there exist integers $l,l'$ with $\tilde n'_i\le  l<l' \le  \tilde n_j$ such that
$\#(E_2\cap(l,l')) \le l'-l-m-1$.
Since $\#(E\cap(l,l')) \ge 2(l'-l)-1$,
we have
$\#(E_1\cap(l,l')) \ge l'-l+m$.

Using Lemma~\ref{Lm:L3a} again, we get integers
$i_1< i'_1=i_2< i'_2$ such that
$E_l\cap[i_l, i'_l]$  is a sampling sequence for   $V_m|_{[i_l,i'_l]}$, $l=1,2$.

Now assume that
\begin{eqnarray*}
f_1 + f_2 &=& \sum_{n\in\bbZ} c_n \varphi_m(\cdot-k), \\
f_1 - f_2 &=& \sum_{n\in\bbZ} c'_n \varphi_m(\cdot-k).
\end{eqnarray*}
Since $E_l\cap[i_l, i'_l]$  is a sampling sequence for   $V_m|_{[i_l,i'_l]}$, $l=1,2$,
we have
\begin{eqnarray*}
  c_i = 0,  && i_1-m\le i\le i'_1-1, \\
  c'_i = 0,  && i_2-m\le i\le i'_2-1.
\end{eqnarray*}
Hence
\[
   c_i \pm c'_i =0,\qquad i_2-m \le i\le i_2-1.
\]
On the other hand, since $f_1, f_2 \ne 0$, there exist some $k_1, k_2$ such that
$c_{k_1}+c'_{k_1}\ne 0$ and $c_{k_2} - c'_{k_2}\ne 0$.
If $k_1$ and $k_2$ lie on the two sides of $i_2-1$, respectively, then  similar arguments as that in the proof of Lemma~\ref{Lm:L4}
  show that both $f_1$ and $f_2$ are separable.

Next we assume that both $k_1$ and $k_2$ are on the same side of $i_2-1$
whenever $c_{k_1}+c'_{k_1}\ne 0$ and $c_{k_2} - c'_{k_2}\ne 0$.
Without loss of generality, we assume that $c_i\pm c'_i=0$ for $i>i_2-1$
and $c_{k_0} + c'_{k_0}\ne 0$,  where $k_0 = \max\{i:\, |c_i + c'_i| + |c_i-c'_i|>0 \}$.
In this case,
\begin{eqnarray*}
f_1 + f_2 &=& \sum_{n\le k_0} c_n \varphi_m(\cdot-k), \\
f_1 - f_2 &=& \sum_{n\le k_0} c'_n \varphi_m(\cdot-k).
\end{eqnarray*}
Hence $f_1(x)=f_2(x)=0$ for $x\ge k_0+m+1$.

Take some $\{y_i:\, 1\le i\le 4m\}\subset [k_0+m+1, k_0+m+2]$.
Let $\tilde E_1 = (E_1\cap (-\infty, k_0+m+1) )\cup \{y_i:\,1\le i\le 2m\}$,
 $\tilde E_2 = (E_2\cap (-\infty, k_0+m+1) )\cup \{y_i:\, 2m+1\le i\le 4m\}$.
Then we have
$(f_1+f_2)|_{\tilde E_1} = 0$
and $(f_1-f_2)|_{\tilde E_2} = 0$.
Moreover, $\tilde E_l\cap [k_0+m+1,k_0+m+2] $ is a sampling sequence for $V_m|_{[k_0+m+1,k_0+m+2]}$, $l=1,2$.

Since $\#(\tilde E_l\cap (k_0+m,k_0+m+1)) \ge 1$ for $l=1$ or $2$,
one of $\tilde E_l\cap [k_0+m,k_0+m+2]$, $l=1,2$, is a sampling sequence for $V_m|_{[k_0+m,k_0+m+2]}$.
Without loss of generality, assume that
$\tilde E_2\cap [k_0+m,k_0+m+2]$ is a sampling sequence for $V_m|_{[k_0+m,k_0+m+2]}$.
Then we have $c'_{k_0} = 0$.

Repeating the previous arguments we get a sequence of integers
\[
  \ldots \le \tilde n_i < \tilde n'_i \le \ldots
\le \tilde n_0<\tilde n'_0 \le k_0+m
\]
such that for each $i\le 0$,
one of $E_l\cap [\tilde n_i, \tilde n'_i]$, $l=1,2$, is a sampling sequence for $V_m|_{[\tilde n_i, \tilde n'_i]}$.
We conclude that there exist some $l<l'\le k_0+m$  such that
$E_1\cap [l,l']$ is a sampling sequence for $V_m|_{[l, l']}$.

Assume on the contrary that $E_1\cap [l,l']$ is not a sampling sequence for $V_m|_{[l, l']}$ for any $l<l'\le k_0+m$.
Then $E_2\cap [\tilde n_i, \tilde n'_i]$ is a sampling sequence for $V_m|_{[\tilde n_i, \tilde n'_i]}$ for all $i\le 0$.
Moreover, by Lemma~\ref{Lm:samp},
\begin{equation}\label{eq:s4:e1}
   \#(E_1\cap(l,l')) \le l'-l+m-1, \qquad \forall l<l'\le k_0+m.
\end{equation}

If $E_2\cap [\tilde n_i, k_0+m+2]$ is a sampling sequence for $V_m|_{[\tilde n_i, k_0+m+2]}$ for all $i\le 0$,
then we have $(f_1-f_2)|_{[\tilde n_i, k_0+m+2]}=0$ for all $i\le 0$. Hence
$f_1(x) - f_2(x)=0$ for $x\le k_0+m+2$.   Therefore, $f_1-f_2=0$, which is impossible.
Hence there is some $i\le 0$ such that
$E_2\cap [\tilde n_i, k_0+m+2]$ is not a sampling sequence for $V_m|_{[\tilde n_i, k_0+m+2]}$.

Observe that
$E_2\cap [\tilde n_i, \tilde n'_i]$
and $E_2\cap [k_0+m, k_0+m+2]$
are sampling sequences for $V_m|_{[\tilde n_i, \tilde n'_i]}$
and $V_m|_{[k_0+m, k_0+m+2]}$, respectively.
By (\ref{eq:s4:e1}), we have
\begin{eqnarray*}
  \#(E_2\cap (\tilde n'_i, k_0+m)) &\ge&  2(k_0+m-\tilde n'_i)-1 - (k_0+m-\tilde n'_i+m-1) \\
  &=& k_0 -\tilde n'_i.
\end{eqnarray*}
Hence
\[
  \#(E_2\cap [\tilde n_i, k_0+m+2]) \ge k_0 - \tilde n_i +2m +2.
\]
By Proposition~\ref{prop:local sampling}, there exist integers $l<l'$ such that $\tilde n_i\le l<l'\le  k_0+m $
and $\#(E_2\cap (l,l'))\le l'-l-m-1$. Hence $\#(E_1\cap (l,l'))\ge l'-l+m$, which contradicts with (\ref{eq:s4:e1}).

It follows that  there exist some $l<l'\le k_0+m$  such that
$E_1\cap [l,l']$ is a sampling sequence for $V_m|_{[l, l']}$.
Since $E_2\cap[k_0+m, k_0+m+2]$ is a sampling sequence for $V_m|_{[k_0+m,k_0+m+2]}$,
we see from the proof of Lemma~\ref{Lm:L3a} that there exist integers $i_1<i'_1=i_2<i'_2$
such that $i_2\le k_0+m$ and
$E_l\cap[i_l, i'_l]$  is a sampling sequence for   $V_m|_{[i_l,i'_l]}$, $l=1,2$.
Hence
\[
  c_i \pm c'_i =0, \qquad i_2-m \le i\le i_2-1.
\]
Since $c_{k_0}+c'_{k_0}\ne 0$, we have $i_2 \le k_0$.
If there is some  $k_1<i_2-m$ such that
\begin{equation}\label{eq:s4:e2}
c_{k_1} + c'_{k_1}\ne 0\quad \mathrm{or}\quad c_{k_1} - c'_{k_1}\ne 0,
\end{equation}
then both $f_1$ and $f_2$ are separable.

In fact, we see from (P1) and (P2) that $E\cap [i, i_2]$ is a sampling sequence for $V_m|_{[i,i_2]}$
provided $i_2-i$ is large enough. Now we see from (\ref{eq:s4:e2})
that $f_l|_{[i,i_2]} \ne 0$, $l=1,2$. Similarly we can show that $f_l|_{[i_2,k_0+m+2]} \ne 0$, $l=1,2$. Hence
both $f_1$ and $f_2$ are separable.

It remains to consider the case of $c_n\pm c'_n=0$ for all $n\le i_2-1$.  In this case, we have
\begin{eqnarray*}
f_1 + f_2 &=& \sum_{i_2\le n\le k_0} c_n \varphi_m(\cdot-k), \\
f_1 - f_2 &=& \sum_{i_2\le n\le k_0} c'_n \varphi_m(\cdot-k),
\end{eqnarray*}
$(f_1+f_2)|_{E_1}=0$ and $(f_1-f_2)|_{E_2}=0$.

Let $E_0 = (E\cap [i_2, k_0+m+2]) \cup F$, where $F\subset (i_2-1,i_2) \cup (k_0+m+1, k_0+m+2)$,
$\# (F\cap (i_2-1, i_2)) = \#(F\cap (k_0+m+1,k_0+m+2)) = m+3$.
By Theorem~\ref{thm:PR}, $E_0$ is a phaseless sampling sequence for $V_m|_{[i_2-1,k_0+m+2]}$.
Since $|f_1(x)|=|f_2(x)|$ for $x\in E_0$ and $f_1 \pm f_2\ne 0$,
 both $f_1$ and $f_2$ are separable.
This completes the proof for $m\ge 2$.

Next we consider the case of $m=1$.
If $k_1=-\infty$ and $k_2=\infty$, then the above arguments also work for $m=1$.

If $k_2<\infty$, then $E\cap [n_{k_2}-1, n]$ is a phaseless sampling sequence for $V_1|_{[n_{k_2}-1, n]}$ for any $n>n_{k_2}$, thanks to Theorem~\ref{thm:PR}.
If $f_1$ or $f_2$ is nonseparable, then
\[
  (f_1+f_2)|_{[n_{k_2}-1, n]} =0\quad or\quad (f_1-f_2)|_{[n_{k_2}-1, n]} =0,\quad n>n_{k_2}.
\]

We conclude that one of $f_1+f_2$ and $f_1-f_2$ must be zero on $[n_{k_2}-1, \infty)$.
In fact, if $(f_1+f_2)|_{[n_{k_2}-1, n]}\ne 0$ for some $n>n_{k_2}$, then
$(f_1-f_2)|_{[n_{k_2}-1, n]} =0$ for all $n>n_{k_2}$.

Suppose that $f_1=f_2$ on $[n_{k_2}-1, \infty)$.
Let $\tilde E_2 = E_2 \cup \{n+x:\, n\ge n_{k_2}, x=1/3,2/3, 1\}$
and $\tilde E = E_1\cup \tilde E_2$.
Then $(f_1-f_2)|_{\tilde E_2} =0$ and $\#(\tilde E\cap[n,n+1]) \ge 3$ for $n\ge n_{k_2}$.

For the case of $k_1>-\infty$, with similar arguments we get $\tilde E=\tilde E_1\cup \tilde E_2$
such that
$(f_1+f_2)|_{\tilde E_1} =0$, $(f_1-f_2)|_{\tilde E_2} =0$,
and $\#(\tilde E\cap[n,n+1]) \ge 3$ for $n\le n_{k_1}-1$.
Now we see from previous arguments that both $f_1$ and $f_2$ are separable.

Finally, suppose that $|f_1(x)|=|f_2(x)|$ on $E$ for some $f_1,f_2\in V_m$
and $E$ meets (P1) and (P2) for $m\ge 2$ or (P1) and (P2') for $m=1$.
We see from the above arguments that there exist integers
$i_k<i'_k\le i_{k+1}<i'_{k+1}$, $k\in\bbZ$ such that $|f_1(x)|=|f_2(x)|$
on $[i_k, i'_k]$, $k\in\bbZ$. By Theorem~\ref{thm:PR}, it is easy to see that
$|f_1(x)|=|f_2(x)|$ for any $x\in\bbR$. This completes the proof.
\end{proof}

\section{Examples}

In this section, we give some examples to illustrate our main results on the characterization of phaseless sampling sequences.
The first one is on the almost phaseless sampling.
\begin{Example}\upshape
Let $N_1$, $N_2$ and $K$ be integers such that $N_1<N_2$.
Define
\[
  x_i = N_1 + \frac{i-1}{K-1} (N_2-N_1),\qquad 1\le i\le K.
\]
It follows from Theorem~\ref{thm:almost} that $E = \{x_i:\, 1\le i\le K\}$ is an almost phaseless sampling sequence for $\VN$ if $K\ge N_2-N_1+m+1$.
\end{Example}

The second example is on the local phaseless sampling.
\begin{Example}\upshape
Let $N_1$, $N_2$ and $K$ be integers such that $N_1<N_2-2$.
Define
\[
  x_i = N_1 + 1 + \frac{i-1}{K-1} (N_2-N_1-2),\qquad 1\le i\le K,
\]
and
\[
  y_i=N_1+\frac{i}{m+1},\quad z_i=N_2-\frac{i}{m+1},\qquad 0\le i\le m.
\]
We see from Theorem~\ref{thm:PR} that $E = \{x_i:\, 1\le i\le K\}\cup \{y_i:\, 0\le i\le m\}\cup \{z_i:\, 0\le i\le m\}$ is a phaseless sampling sequence for $\VN$ if $K\ge 2(N_2-N_1)-3$.
\end{Example}

The third example is on the global phaseless sampling, which is a simple consequence of Theorem~\ref{thm:global}.
It illustrates the difference between the case of $m=1$ and the other case of $m>1$.

\begin{Example}\upshape
Let $\alpha>\beta\ge 0$ be  constants.
\begin{enumerate}
\item For $m\ge 2$,  $\{n\alpha+\beta:\,n\in\bbZ\}$ is a phaseless sampling sequence for $V_m$ if and only if $0<\alpha<1/2$.

\item
For  $m=1$, $\{n\alpha+\beta:\,n\in\bbZ\}$ is a phaseless sampling sequence for $V_m$ if and only if $0<\alpha < 1/2$
or $\alpha=1/2$ and $\beta=0$.
\end{enumerate}

\end{Example}


\end{document}